\documentclass[12pt,makeidx]{amsart}
\usepackage{amssymb,mathrsfs, amsmath, amsthm}
\usepackage{url,titletoc,enumitem}

\usepackage[usenames,dvipsnames]{xcolor}
\definecolor{darkblue}{rgb}{0.0, 0.0, 0.55}
\usepackage[pagebackref,colorlinks,linkcolor=BrickRed,citecolor=OliveGreen,urlcolor=darkblue,hypertexnames=true]{hyperref}

\renewcommand{\subset}{\subseteq}
\renewcommand{\emptyset}{\varnothing}

\makeatletter
\newcommand{\mycontentsbox}{%
{\addtolength{\parskip}{-2.3pt}
\small\tableofcontents}}
\def\enddoc@text{\ifx\@empty\@translators \else\@settranslators\fi
\ifx\@empty\addresses \else\@setaddresses\fi
\newpage\mycontentsbox\newpage\printindex}
\makeatother

\newtheorem{theorem}{Theorem}[section]

\newtheorem{prop}[theorem]{Proposition}

\newtheorem{remark}[theorem]{Remark}

\newtheorem{example}[theorem]{Example}
\newtheorem{thm}[theorem]{Theorem}
\newtheorem{lem}[theorem]{Lemma}

\newtheorem*{lemma*}{Lemma}

\def\beq{\begin{equation}}
\def\eeq{\end{equation}}

\def\ben{\begin{enumerate}}
\def\een{\end{enumerate}}

\def\tY{\tilde{Y}}
\def\tB{\tilde{B}}
\def\hB{\hat{B}}
\def\hb{\hat{b}}

\def\hom{\mathfrak{H}}

\def\ssec{\subsection}
\def\sssec{\subsubsection}

\numberwithin{equation}{section}

\def\cD{ {{\mathcal D}}}
\def\cH{ {{\mathcal H}}}

\def\cP{ {{\mathcal P}}}

\def\R{ {\mathbb{R}} }
\def\C{ {\mathbb{C}} }

\def\cC{ {\mathcal C} }
\def\cD{ {\mathcal D} }
\def\cE{ {\mathcal K} }

\def\cH{ {\mathcal H} }

\def\cP{{\mathcal P}}

\def\cS{{\mathcal S} }
\def\cT{{\mathcal T}}

\def\cP{{\mathcal P}}

 \def\bs{\bigskip}

\def\cS{{\mathcal S}}

\newcommand{\df}[1]{{\bf{#1}}{\index{#1}}}

\def\arv{\partial^{\operatorname{Arv}}}
\def\euc{\partial^{\operatorname{Euc}}}

\def\abs{\partial^{\operatorname{free}}}

\def\comat{\operatorname{co}^{\operatorname{mat}}}

\def\ssplus1{\cD_{\cS}^{g+1}}

\def\cE{\mathcal{E}}

\def\smnrg{{SM_n(\R)^g}}
\def\smnrgp{{SM_n(\R)^{g+1}}}
\def\smmrgp{{SM_n(\R)^{g+1}}}

\def\smnkg{{SM_n(\R)^g}}

\def\smnonerg{{SM_{n_1}(\R)^g}}

\def\smntworg{{SM_{n_2}(\R)^g}}

\def\smrg{{SM(\R)^g}}
\def\smrgp{{SM(\R)^{g+1}}}

\def\smdrg{{SM_d(\R)^g}}
\def\smdrgp{{SM_d(\R)^{g+1}}}

\def\smnirg{{SM_{n_i}(\R)^g}}
\def\smxr[#1]{SM_{#1}(\R)}
\def\smxrg[#1]{SM_{#1}(\R)^g}

\def\bbN{ {\mathbb N}}

\def\mathand{{\qquad \mathrm{and} \qquad}}

\makeindex

\date{\today}

\begin{document}

\setcounter{tocdepth}{3}
\contentsmargin{2.55em} 
\dottedcontents{section}[3.8em]{}{2.3em}{.4pc} 
\dottedcontents{subsection}[6.1em]{}{3.2em}{.4pc}
\dottedcontents{subsubsection}[8.4em]{}{4.1em}{.4pc}

\title[Arveson boundary of free quadrilaterals]{The Arveson boundary of a Free Quadrilateral is given by a noncommutative variety}

\author[E. Evert]{Eric Evert${}^{1}$}
\address{Eric Evert, Group Science, Engineering and Technology\\
 KU Leuven Kulak \\
  E. Sabbelaan 53, 8500 Kortrijk, Belgium \\
  and
  \newline
   Electrical Engineering ESAT/STADIUS\\
  KU Leuven \\
  Kasteelpark Arenberg 10, 3001 Leuven, Belgium
   }
   \email{eric.evert@kuleuven.be}
\thanks{${}^1$Research supported by the NSF grant
DMS-1500835}

\begin{abstract}
Let $ \smnrg$ denote $g$-tuples of $n \times n$ real symmetric matrices and set $\smrg = \cup_n \smnrg$. A free quadrilateral is the collection of tuples $X \in SM(\R)^2$ which have positive semidefinite evaluation on the linear equations defining a classical quadrilateral. Such a set is closed under a rich class of convex combinations called matrix convex combination. That is, given elements $X=(X_1, \dots, X_g) \in \smnonerg$ and $Y=(Y_1, \dots, Y_g) \in \smntworg$ of a free quadrilateral $\mathcal{Q}$, one has
\[
V_1^T X V_1+V_2^T Y V_2 \in \mathcal{Q} 
\]
for any contractions $V_1:\R^n \to \R^{n_1}$ and $V_2:\R^n \to \R^{n_2}$ satisfying $V_1^T V_1+V_2^T V_2=I_n$. These matrix convex combinations are a natural analogue of convex combinations in the dimension free setting. 

A natural class of extreme point for free quadrilaterals is free extreme points: elements of a free quadrilateral which cannot be expressed as a nontrivial matrix convex combination of elements of the free quadrilateral. These free extreme points serve as the minimal set which recovers a free quadrilateral through matrix convex combinations. 

In this article we show that the set of free extreme points of a free quadrilateral is determined by the zero set of a collection of noncommutative polynomials. More precisely, given a free quadrilateral $\mathcal{Q}$, we construct noncommutative polynomials $p_1,p_2,p_3,p_4$ such that a tuple $X \in SM (\R)^2$ is a free extreme point of a $\mathcal{Q}$ if and only if $X \in \mathcal{Q}$ and $p_i(X) =0 $ for $i=1,2,3,4$ and $X$ is irreducible.

In addition we establish several basic results for projective maps of free spectrahedra and for homogeneous free spectrahedra. Namely we show that that the image of a free extreme point under an invertible projective map is again a free extreme point and we extend a kernel condition for a tuple to be a free extreme point to the setting of homogeneous free spectrahedra. 
\end{abstract}

\maketitle

\section{Introduction}

This article studies the free (noncommutative) extreme points of a natural generalization of a quadrilateral to the dimension free setting, namely free quadrilaterals. Such a set arises by considering those tuples of symmetric matrices with arbitrary size which have positive semidefinite evaluation on the linear equations defining a classical quadrilateral. Our main result is a classification of the set of free extreme points of a free quadrilateral as the zero set of a collection of noncommutative polynomials. In pursuit of this result, we establish several results for homogeneous free spectrahedra and for linear and projective transformations of (homogeneous) free spectrahedra. 

The study of extreme points for dimension free sets was first initiated in Arveson's seminal works \cite{A69,A72} in the infinite dimensional setting of operator systems and has since been considered by many authors \cite{Ham79,MS98,FHL18,DK+}. Here, the main question was to determine if an operator system is completely determined by its boundary representations, objects which serve as a natural class of extreme point in this infinite dimensional context. This question was affirmatively answered in the seperable setting nearly forty years after its conception by Arveson \cite{A08} following on work of Dritchel and McCullough \cite{DM05} and Agler \cite{A88}. A few years later, Davidson and Kennedy \cite{DK15} settled the issue with a positive answer in the fully generality of Arveson's original question. 

Arveson's original question in infinite dimensions has a natural translation to the finite dimensional setting of matrix convex sets: collections of matrix tuples of all sizes which are closed under matrix (dimension free) convex combinations. Here the goal is to find the smallest class of extreme points for a matrix convex set which recovers the set through matrix convex combinations. Free extreme points, the natural analogue of Arveson's boundary representations in the finite dimensional setting, were introduced by Kleski \cite{Kls14}. While it is known that not all matrix convex sets have free extreme points, e.g. see \cite{E18,K+}, it has recently been shown that for matrix convex sets such as free quadrilaterals which arise as the positivity domain of a noncommutative polynomial, free extreme points span \cite{EH19}. Other works considering matrix convex sets and free extreme points include \cite{WW99,F00,F04,FNT17,DDSS17,PSS18,EHKM18,PP+}.

In addition to matrix convexity, the results in this article have a relationship to the burgeoning areas of noncommutative function theory and noncommutative real algebraic geometry \cite{Voi10,KVV14,MS11,Pop10,AM15,BB07,JKMMP20,SSS18,HKV+}. Here one studies noncommutative functions and polynomials whose arguments are tuples of matrices of all sizes. In particular we show that the set of free extreme points of a free quadrilateral may be expressed as the intersection of the zero set of a collection of noncommutative polynomials, i.e. a noncommutative variety, with the free quadrilateral. That is, given a free quadrilateral $\mathcal{Q}$, we construct noncommutative polynomials $p_1,p_2,p_3,p_4$ such that a tuple of symmetric $n \times n$ matrices $X = (X_1,X_2)$ is, up to a technical assumption, a free extreme point of $\mathcal{Q}$ if and only if $p_1(X)=p_2(X)=p_3(X) = p_4(X) =0$ and $X \in \mathcal{Q}$. 

As an example, for the free quadrilateral with defining relations 
\[
\begin{array}{rl}
I+3X_1+2X_2 \succeq 0 \qquad \quad & \qquad I-\frac{X_1}{2} + 3 X_2 \succeq 0  \\
I-X_1-X_2 \succeq 0 \qquad \quad & \qquad I-\frac{3X_1}{2} - 2X_2 \succeq 0,
\end{array}
\]
one has that $X=(X_1,X_2) \in SM_n (\R)^2$ is a free extreme point of $\mathcal{Q}$ if and only if $X$ is irreducible, $X \in \mathcal{Q}$, and
\[
\begin{array}{lclcl}
p_1 (X) & = & 832 + 3098 X_1 + 656 X_2 - 1547 X_1  X_1 - 4228 X_1  X_2  \\
& & - 4228 X_2  X_1 - 3003 X_1  X_1  X_1 - 4522 X_1 X_2  X_1  &  =  & 0 \\
 p_2 (X) & = & 3465 X_1  X_2 - 3465 X_2  X_1 - 4719 X_1  X_1  X_2 \\
 & &  - 7106 X_1 X_2  X_2 + 4719 X_2  X_1 X_1 + 7106 X_2 X_2 X_1 & = &  0 \\
 p_3 (X) & = & 596 + 24 X_1 - 649 X_2 + 1995 X_1 X_2 + 1995 X_2 X_1\\
 & &  - 5957 X_2 X_2  + 3003 X_2  X_1 X_2 + 4522 X_2 X_2 X_2 & = & 0 \\
  p_4 (X) & = & 5554 + 16521 X_1 + 5554 X_2 - 10207 X_1 X_1 - 15863 X_1  X_2 \\
 & & - 19328 X_2  X_1  - 6644 X_2  X_2 - 13728 X_1 X_1 X_1 \\
 & & - 4719 X_1 X_1 X_2 - 20672 X_1 X_2 X_1 - 7106 X_1 X_2 X_2 &  = & 0.
\end{array}
\]

In the remainder of the introduction we introduce our basic definitions and notation and give a formal statement of our main results. 

	\ssec{Free convex sets and free spectrahedra}

	\sssec{Matrix Convex Sets}
	
	Let $\smnrg$ \index{$\smnrg$} denote the set of $g$-tuples of real symmetric $n \times n$ matrices and define $\smrg:=\cup_n \smnrg$ . That is, an element $X \in \smnrg$ is a tuple 
	\[
	X=(X_1, X_2, \dots X_g)
	\]
	where $X_i \in \R^{n \times n}$ and $X_i=X_i^T$ for each $i=1,2, \dots, g$. Similarly we let $M_{m \times n} (\R)^g$ \index{$M_{m \times n} (\R)^g$} and $M_{n} (\R)^g$ denote the sets of $g$-tuples of $m \times n$ matrices and $g$-tuples of $n \times n$ matrices with real entries, respectively. Given a matrix tuple $X \in \smnrg$ and matrices $V \in M_{m \times n} (\R)$ and $W \in M_{n \times p} (\R)$, we let $V X W$ denote the coordinate wise product 
	\[
	V X W = (V X_1 W, \dots, V X_g W) \in M_{m \times p} (\R)^g.
	\]
	
	Given a finite collection of tuples $\{X^i\}_{i=1}^\ell$ with $X^i \in \smnirg$ for each $i=1,\dots, \ell$, a \df{matrix convex combination of $\{X^i\}_{i=1}^\ell$} is a finite sum of the form
	\[
	\sum_{i=1}^\ell V_i^T X^i V_i \quad \quad \quad \quad \sum_{i=1}^\ell V_i^T V_i=I_n.
	\]
where $V_i \in M_{n_i \times n} (\R)$ for each $i=1,2, \dots, \ell$. We emphasize that the matrix tuples $X^i$ can be of different sizes. That is, the $n_i$ need not be equal.

	A set $K \subset \smrg $ is \df{matrix convex} if it is closed under matrix convex combinations. The \df{matrix convex hull} of a set $K \subset \smrg$, denoted $\comat (K)$  is the set of all matrix convex combinations of elements of $K$. Equivalently, the matrix convex hull of $K$ is the smallest matrix convex set which contains $K$. 
	
	For a subset $K \subset \smrg$, we call the set $K(n):= K \cap \smnrg$ the set $K$ at level $n$. Say $K$ is \df{closed} if $K(n)$ is closed for each $n$ and say $K$ is \df{bounded} if there exists a constant $C > 0$ such that $C - \sum_{i=1}^g X_i^2 \succeq 0$ for all $X \in K$. In the case that $K$ is matrix convex, one may show that $K$ is bounded if and only if $K(1)$ is bounded, e.g. see \cite{PSS18}.

		\subsubsection{Free Spectrahedra and Linear Matrix Inequalities} 
A prototypical example of matrix convex sets are those defined by a linear matrix inequality, namely free spectrahedra. Given a $g$-tuple $A$ of $d \times d$ real symmetric matrices, let $\Lambda_A$ denote the \df{homogeneous linear pencil}
\[
\Lambda_A (x)=A_1 x_1+ \cdots + A_g x_g
\]
and let $L_A$ denote the \df{monic linear pencil}
\beq
\label{eq:MonicLinPencilDef}
L_A (x)=I_d + A_1 x_1+ \cdots +A_g x_g .
\eeq
For a positive integer $n \in \bbN$ and a $g$-tuple $X \in \smnkg$, the \df{evaluation} of the monic linear pencil $L_A$ on $X$ is
\[
L_A (X):=I_{dn}+\Lambda_A (X)=I_{dn}+A_1 \otimes X_1 + \cdots + A_g \otimes X_g
\index{$L_A(X)$}
\]
where $\otimes$ denotes the Kronecker product. A \df{linear matrix inequality} is an inequality of the form $L_A (X) \succeq 0$. 

The \df{free spectrahedron at level $n$}, denoted $\cD_A (n)$, is the set
\[
\cD_A (n)=\{X \in \smnrg |\ L_A (X) \succeq 0 \} \index{$\cD_A (n)$},
\]
and the \df{free spectrahedron} $\cD_A$ is the dimension free set $\cD_A:= \cup_n \cD_A(n) \subset \smrg$. In other words,
\[
\cD_A=\{X \in \smrg |\ L_A (X) \succeq 0 \} \index{$\cD_A$}.
\]
It straightforward to show that all free spectrahedra are matrix convex.

As a consequence of \cite{EW97}, every closed matrix convex set may be expressed as a (perhaps infinite) intersection of free spectrahedra. In addition, every matrix convex set which is the positivity domain of a noncommuative polynomial is a free spectrahedron \cite{HM12}.

\sssec{Free polytopes and free quadrilaterals}

A special class of free spectrahedra are those which may be defined by a tuple of matrices which is simultaneously diagonalizable. Such free spectrahedra are called \df{free polytopes}. These free polytopes serve as a natural generalization of linear programming to the free setting. 

Of particular interest in this article are free quadrilaterals. Say a free spectrahedron $K \subset SM (\R)^2$ is a \df{free quadrilateral} if  $K(1) \subset \R^2$ is a bounded quadrilateral in the classical sense and there is a tuple $A=(A_1,A_2)$ of $4 \times 4$ real diagonal matrices such $K = \cD_A$. Equivalently, a free spectrahedron $K \subset SM (\R)^2$ is a free quadrilateral if $K$ is bounded and has a minimal defining tuple $A \in SM_4 (\R)^2$ of $4 \times 4$ diagonal matrices. A formal definition of a minimal defining tuple is given in Section \ref{sec:MinDefining}.

	\subsection{Extreme Points of Free Spectrahedra} There are many notions of extreme points for matrix convex sets. Two particularly notable types are classical (Euclidean) and free (Arveson) extreme points, which respectively represent the least and most restricted types of extreme points for matrix convex sets. 

	Given an element $X \in K(n)$ of a matrix convex set $K \subset \smrg$, say $X$ is a \df{Euclidean extreme point} of $K$ if $X$ is a classical extreme point of $K(n)$, i.e. if $X$ cannot be expressed as a nontrivial convex combination of elements of $K(n)$. We let $\euc K$ \index{$\euc K$} denote the set of Euclidean extreme points of $K$.

While Euclidean extreme points are natural to consider when working with classical convex combinations, the additional freedom in matrix convex combinations often allow a Euclidean extreme point of a given matrix convex set to be expressed as a nontrivial matrix convex combination of elements of the set. In contrast, a free extreme point cannot be expressed as a nontrivial matrix convex combination.

Before giving a formal definition for free extreme points, we give a brief definition. Given tuples $X,Y \in \smnrg$, if  there is an orthogonal (i.e. a real valued unitary) matrix $U$ so that
	\[
	U^T X U =(U^T X_1 U, \dots, U^T X_g U)= (Y_1, \dots, Y_g)
	\]
	then we say $X$ and $Y$ are \df{unitarily equivalent}. A subset $E \subset \cD_A$ of a free spectrahedron is \df{closed under unitary equivalence} if $X \in E$ and $X$ is unitarily equivalent to $Y$ implies $Y \in E$. 
	
	A tuple $X \in K(n)$ is a \df{free extreme point} of $K$ if whenever $X$ is written as a  matrix convex combination 
	\[
	X=\sum_{i=1}^k V_i^T Y^i V_i \quad \quad \quad \quad \sum_{i=1}^k V_i^T V_i =I_n
	\]
	with $Y^i \in K(n_i)$ and $V_i \neq 0$ for each $i$, then for all $i$ either $n_i=n$ and  $X$ is unitarily equivalent to $Y^i$ or $n_i > n$ and there exists a tuple $Z^i \in K$ such that  $X \oplus Z^i$ is unitarily equivalent to $Y^i$. We let $\abs K$ \index{$\abs K$} denote the set of free extreme points of $K$.
	
Free extreme points are a natural type of extreme point for free spectrahedra in that they are the minimal subset of a bounded free spectrahedron which recovers the free spectrahedron through matrix convex combinations.
	
		\begin{thm} \cite[Theorem 1.1]{EH19}
		\label{thm:AbsSpanMin}
		Let $A \in \smdrg$ such that $\cD_A$ is a bounded free spectrahedron. Then $\cD_A$ is the matrix convex hull of its free extreme points. Furthermore, if $E \subset \cD_A$ is a set of irreducible tuples which is closed under unitary equivalence and whose matrix convex hull is equal to $\cD_A$, then $E$ must contain the free extreme points of $\cD_A$.
	\end{thm}
	  
	  	\sssec{Irreducibility of matrix tuples}
Free extreme points are irreducible as tuples of matrices, a notion we now define.	Given a matrix $M \in \R^{n \times n}$, a subspace $N \subset \R^n$ is a \df{reducing subspace} if both $N$ and $N^\perp$ are invariant subspaces of $M$. That is, $N$ is a reducing subspace for $M$ if $MN \subseteq N$ and $MN^\perp \subseteq N^\perp$. A tuple $X \in \smnrg$ is \df{irreducible} (over $\R$) if the matrices $X_1, \dots, X_g$ have no common reducing subspaces in $\R^n$; a tuple is \df{reducible} (over $\R$) if it is not irreducible (over $\R$).

\ssec{Free extreme points and the Arveson boundary} \ 
Free extreme points are closely related to the classical dilation theoretic Arveson boundary. Say a tuple $X$ in a free spectrahedron $\cD_A$ is an \df{Arveson extreme point} of $\cD_A$ if
	\beq
	\label{eq:arvdef}
	Y= \begin{pmatrix} X & \beta \\
	\beta^T & \gamma \end{pmatrix}
	\in \cD_A
	\eeq
	implies $\beta=0$. The set of Arveson extreme points of a free spectrahedron $\cD_A$ is called the \df{Arveson boundary} of $\cD_A$ and is denoted by $\arv \cD_A$ \index{$\arv \cD_A$}
	
	The following theorem illustrates the relationship between the free, Arveson, and Euclidean extreme points of a free spectrahedron.
	
	\begin{thm}
		\label{thm:ArvImpliesEuc}
		Let $A \in \smdrg$ and let $\cD_A$ be the associated free spectrahedron.
		\ben 
		\item \label{it:FreeIsArvIrred} A tuple $X \in \cD_A (n)$ is a free extreme point of $\cD_A$ if and only if $X$ is an irreducible Arveson extreme point of $\cD_A$.  
		\item \label{it:ArvIsEuc} If $X \in \cD_A (n)$ is an Arveson extreme point of $\cD_A$, then $X$ is a Euclidean extreme point of $\cD_A$.
		
		\item \label{it:KerCondition} A tuple $X \in \cD_A$ is an Arveson extreme point of $\cD_A$ if and only if the only tuple $\beta \in M_{n \times 1} (\R)^g$ satisfying
		\beq
		\label{eq:ArvesonKernelClassical}
		\ker L_A (X) \subseteq \ker \Lambda_A (\beta^T) \qquad \mathrm{is \ } \beta =0.
		\eeq
		
		\een
	\end{thm}
	\begin{proof}

	\cite[Theorem 1.1]{EHKM18} proves \eqref{it:FreeIsArvIrred} and \eqref{it:ArvIsEuc} when working over $\C$. The \cite{EHKM18} proof of \eqref{it:ArvIsEuc} can be used over $\R$ without change. An adapted proof of \eqref{it:FreeIsArvIrred} which works over $\R$ is given by \cite[Theorem 1.2]{EH19}. Item \eqref{it:KerCondition} is proved as \cite[Lemma 2.1 (3)]{EH19}.  
	\end{proof}
	
	\ssec{Main Results}
\sssec{Homogeneous Free spectrahedra and Projective maps}

Key ingredients in our study of free quadrilaterals are homogeneous free spectrahedra and projective maps of free spectrahedra, and we develop several basic results in each of these directions. Given a $g+1$-tuple $(A_0,A)=(A_0,A_1,\dots,A_g) \in \smdrgp$, we define a homogeneous free spectrahedron to be the set of $(X_0,X)=(X_0,X_1,\dots,X_g) \in \smrgp$ which satisfy $\Lambda_{(A_0,A)} (X_0,X) \succeq 0$. 

We focus our study on the natural class of homogeneous free spectrahedra which contain the tuple $(1,0,\dots,0) \in \R^{g+1}$ as an interior point. We say such a homogeneous free spectrahedron is \textit{positive} and, under light assumptions, we show that a homogeneous free spectrahedron is positive if and only if $A_0 \succ 0$. In addition we show that homogeneous free spectrahedra satisfying these assumptions arise naturally as the \textit{homogenization} of a free spectrahedron. Our main result for positive homogeneous free spectrahedra is an analogue of Theorem \ref{thm:ArvImpliesEuc} \eqref{it:KerCondition}. That is, we establish a kernel condition in the spirit of equation \eqref{eq:ArvesonKernelClassical} for a tuple to be an Arveson extreme point of a positive homogeneous free spectrahedron. See Section \ref{sec:HomSpecBasics} for formal definitions and details.  

Classically, a projective map on $\R^g$ may be described as a map which is linear on the projective space $P(\R^g)$. Informally, a projective map between free spectrahedra $\cD_A$ and $\cD_B$ is a map which is linear between the homogenizations of $\cD_A$ and $\cD_B$. In Theorem \ref{thm:ArvMappings}, we show that {\it if $\cP$ is an invertible projective map which maps a bound free spectrahedron $\cD_A$ onto a bound free spectrahedron $\cD_B$, then $\cP$ maps the Arveson boundary of $\cD_A$ onto the Arveson boundary of $\cD_B$.} In addition we  show that for any two free quadrilaterals $\mathcal{Q}_1$ and $\mathcal{Q}_2$, {\it there exists an invertible projective map $\cP$ such that $\cP (\mathcal{Q}_1) = \mathcal{Q}_2$} thus extending a well known classical result to the free setting. Projective maps of free spectrahedra are discussed in Section \ref{sec:LinearAndProjectiveTransformations}.

	\sssec{Noncommutative polynomials and varieties}
	
	A \df{noncommutative polynomial} in $g$ variables is a finite sum of the form
	\[
	p(x) = \sum_{w} \alpha_w w \in \R \langle x \rangle
	\]
	where $\alpha_w \in \R$, and the $w$ are words in the noncommutative variables $x = (x_1, \dots, x_g)$. The \df{degree} of a word is the length of the word, and the degree of a noncommutative polynomial is given by its highest degree word. For example, $x_1 x_2 + 3 x_2 x_1 x_2$ is a noncommutative polynomial of degree $3$. As with linear matrix inequalities, the \df{evaluaion} a noncommutative polynomial $p(x)$ on a matrix tuple $X \in \smrg$ is obtained by replacing $x_i \to X_i$.
	
	A set free set $\Gamma \subset \smrg$ is a \df{noncommutative variety} if it is the zero set of a finite collection of noncommutative polynomials. That is, if $\Gamma$ is a noncommutative variety, then there exists a finite collection of noncommutative polynomials $\{p_1, \dots, p_\ell\}$ such that
	\[
	\Gamma = \{X \in \smrg | \ p_i (X) = 0 \ \mathrm{for \ all\ } i=1,\dots,\ell\}.
	\]
	One may readily verify that a noncommutative variety is closed under direct sums and simultaneous unitary conjugation. 

 Our main result regarding noncommutative polynomials is that the Arveson boundary of a free quadrilateral is determined by a noncommutative variety.
	\begin{thm}
\label{thm:ArvIsVariety}
Let $\cD_A$ be a free quadrilateral. Then there exist noncommutative polynomials polynomials $p_1,p_2,p_3,p_4$ in the noncommuting variables $x_1$ and $x_2$ of degree no more than three such that $X \in \arv \cD_A$ if and only if $X \in \mathcal{Q}$ and $p_i (X)=0$ for each $i=1,\dots,4$. 
\end{thm}
The proof of Theorem \ref{thm:ArvIsVariety} is constructive, and a Mathematica notebook which computes these noncommutative polynomials for given free quadrilaterals is available in an online appendix: \newline \url{https://github.com/NCAlgebra/UserNCNotebooks/tree/master/Evert/FreeQuadrilaterals}.
	
\subsubsection{Arveson Boundaries of other free spectrahedra}

Aside from free quadrilaterals, a small number of free spectrahedra are known to have Arveson boundary given by a noncommutative variety; however, these examples are comparatively quite simple. A first example is the free cube in $g$ variables, i.e. the set $\cC^g$ defined by
\[
\cC^g = \{X \in SM (\R)^g | \ X_i^2 \preceq I \mathrm{\ for \ } i=1,\dots,g\}.
\]
As shown by \cite[Proposition 7.1]{EHKM18}, the free cube has Arveson boundary
\[
\arv \cC^g = \{X \in SM (\R)^g | \ X_i^2  - I =0 \mathrm{\ for \ } i=1,\dots,g\},
\] 
hence $\arv \cC^g$ is a noncommutative variety.

Other examples include free simplices and the wild disc. Say a bounded free spectrahedron $\cD_A$ is a free simplex in $g$ variables if there exists a collection $\{a^j\}_{j=1}^{g+1} \subset \R^g$ such that 
\[
\cD_A = \cD_{\oplus_{j=1}^{g+1} a^j}.
\]
In words, a free simplex in $g$ variables is a free spectrahedron whose minimal defining tuple is a $g$ tuple of commuting $g+1 \times g+1$ matrices. For $1 \leq i < j \leq g$, let $c_{ij} (x) = x_i x_j - x_j x_i$ be the commutant of $x_i$ with $x_j$, and for $i=1,\dots,g+1$, define\footnote{The $L_{a^j} (x)$ may appear in this product in any order without impacting the resulting noncommutative variety since $c_{ij} (X)=0$ for all  $1 \leq i < j \leq g$ enforces that $X$ is a tuple of commuting matrices.} $p_i (x) = \Pi_{j \neq i} L_{a^j} (x)$. Using \cite[Theorem 6.5]{EHKM18} one may show that $X \in \smrg$ is an Arveson extreme point of $\cD_{\oplus_{j=1}^{g+1} a^j}$ if and only if 
\[
c_{ij} (X) = 0 \mathrm{\ for \ all\ } 1\leq i<j \leq g \mathand p_i (X)=0  \mathrm{\ for \ all\ }  i=1,\dots,g.
\]

Similarly, \cite[Proposition 7.5]{EHKM18} has the consequence that the Arveson boundary of the wild disc is a noncommutative variety. The wild disc is the free spectrahedron with defining pencil
\[
L_A (x_1,x_2) = \begin{pmatrix}
1 + x_1 & x_2 \\
x_2 & 1-x_1
\end{pmatrix}.
\]
A tuple $X = (X_1,X_2) \in SM (\R)^2$ is an Arveson extreme point of the wild disc if and only if $I-X_1^2-X_2^2 = 0$ and $X_1 X_2- X_2 X_1=0$. 

To the author's knowledge, there is no free spectrahedron whose Arveson boundary is known to not be a noncommutative variety; however the author is also unaware of other notable examples of free spectrahedra whose Arveson boundary is known to be a noncommutative variety. Limited numerical evidence suggests that free spectrahedra having Arveson boundary equal to a noncommutative variety are rare.

\subsection{Reader's guide}

Section \ref{sec:HomSpecBasics} introduces homogeneous free spectrahedra and establishes several basic results for them. The main result of the section is Theorem \ref{thm:HomogenousArvesonKernelCondition} which extends the kernel condition given in Theorem \ref{thm:ArvImpliesEuc} \eqref{it:KerCondition} for a tuple to be an Arveson extreme point to the setting of homogeneous free spectrahedra. Section \ref{sec:LinearAndProjectiveTransformations} introduces and examines projective and linear transformations for free and homogeneous free spectrahedra. The main result in this section is Theorem \ref{thm:ArvMappings} which shows that invertible projective mappings of bounded free spectrahedra map Arveson boundary onto Arveson boundary. Finally in Section \ref{sec:QuadrilateralArvesonBoundary} we prove our main result, Theorem \ref{thm:ArvIsVariety}. In addition we show that if $\cD_A$ is a free quadrilateral, then $\cD_A$ is the matrix convex hull of its Arveson extreme points at level $2$. 

The paper includes an appendix where we prove that given any two free quadrilaterals $\cD_A$ and $\cD_B$, there exists an invertible projective mapping which maps $\cD_A$ onto $\cD_B$.  This result is well known in the setting of classical quadrilaterals but requires an adaptation to the free setting.

\ssec{Acknowledgements}

The author thanks J. William Helton for helpful discussions related to this topic and for comments on the original version of the manuscript.

\section{Homogeneous Free Spectrahedra}
\label{sec:HomSpecBasics}

Homogeneous free spectrahedra play an important role in this article. Given a $g+1$ tuple $(A_0,A_1,\dots,A_g)=(A_0,A)$ of $d \times d$ symmetric matrices, we define a \df{homogeneous free spectrahedron at level $n$}, denoted $\cH_{(A_0,A)} (n)$, by
\[
\cH_{(A_0,A)} (n)=\{(X_0,X_1,\dots,X_g) \in \smnrgp |\ \Lambda_{(A_0,A)} (X_0,X) \succeq 0 \}. \index{$\cD_A (n)$}
\]
The corresponding \df{homogeneous free spectrahedron} $\cH_{(A_0,A)}$ is the set $\cup_n \cH_{(A_0,A)} (n) \subset \smrg$. In notation,
\[
\cH_{(A_0,A)}=\{(X_0,X) \in \smrgp |\ \Lambda_{(A_0,A)} (X_0,X) \succeq 0 \} \index{$\cD_A$}.
\]

We will often make us of the fact that one may conjugate a tuple $(A_0,A)$ with an invertible matrix without changing the resulting homogeneous free spectrahedron. 
\begin{lem}
\label{lem:homspectuplequiv}
Let $(A_0,A) \in \smdrgp$ and let $V \in M_d (\R)$ be any invertible matrix. Then 
\[
\cH_{(A_0,A)} = \cH_{V^T (A_0,A) V}. 
\]
\end{lem}
\begin{proof}
For any integer $n$ and any tuple $(X_0,X) \in \smnrgp$ and any invertible $V \in M_d (\R)$ one has
\[
\Lambda_{V^T (A_0,A) V} (X_0,X) = (V^T \otimes I_n) \big(\Lambda_{(A_0,A) } (X_0,X)\big) (V \otimes I_n),
\]
hence $\Lambda_{V^T (A_0,A) V} (X_0,X) \succeq 0$ if and only if $\Lambda_{(A_0,A) } (X_0,X) \succeq 0$. 
\end{proof}

The homogeneous free spectrahedra we consider typically arise as the homogenization of a free spectrahedron, hence why the $A_0$ and $X_0$ terms in the $g+1$ tuples above are treated specially. In order to well define such a homogenization, we must first introduce the notion of minimal defining tuples for free and homogeneous free spectrahedra.

\sssec{Minimal defining tuples}
\label{sec:MinDefining}

Given a $g+1$-tuple $(A_0,A) \in \smdrgp$, if there is an orthogonal matrix $U \in M_d (\R)$ such that $U^T (A_0,A) U = (A_0^1,A^1) \oplus (A_0^2,A^2)$, then the linear pencils $\Lambda_{(A_0,A)} (x)$ and $ \big(\Lambda_{(A_0^1,A^1)} \oplus \Lambda_{(A_0^2,A^2)}\big) (x)$ define the same homogeneous free spectrahedron. For $i=1,2$, the linear pencil $\Lambda_{(A_0^i,A^i)}$ is a \df{subpencil} of $\Lambda_{(A_0,A)}$, and if 
$
\cH_{(A_0,A)} = \cH_{(A_0^i,A_0^i)},
$
then $\Lambda_{(A_0^i,A^i)}$ is a \df{defining subpencil} for $\cH_{(A_0,A)}$. Say the tuple $(A_0,A)$ is a \df{minimal defining tuple} for $\cH_{(A_0,A)}$ if for all invertible $V \in M_d(\R)^g$, there no proper subpencil of $\Lambda_{V^T(A_0,A)V}$ which is a defining subpencil for $\cH_{(A_0,A)}$.

For nonhomogeneous free spectrahedra, defining subpencils are defined analogously. A tuple $A \in \smdrg$ is a \df{minimal defining tuple} for a free spectrahedron $\cD_A$ is there is no proper subpencil of $L_A$ which is a defining pencil for $\cD_A$.  Also see \cite{HKM13,Zal17} for details on minimal defining tuples.

\sssec{Homogenizations of free spectrahedra}

Given a free spectrahedron $\cD_A$, the \df{homogenization} of $\cD_A$, denoted $\hom (\cD_A)$, is the homogeneous free spectrahedron $\cH_{(I,\check{A})}$, where $\check{A}$ is any minimal defining tuple for $\cD_A$. The fact that a homogenization is well defined is a consequence of \cite[Theorem 3.12]{HKM13}. Without requiring the defining tuple to be minimal, a homogenization is not necessarily well defined, e.g. see the upcoming Remark \ref{rem:BadHomogenization}.

Similarly, we often consider the ``nonhomogeneous" part of homogeneous free spectrahedra. Given a homogeneous free spectrahedron $\cH_{(A_0,A)})$, let $\hom^{-1} (\cH_{(A_0,A)})$ denote the set
\[
\hom^{-1} (\cH_{(A_0,A)}) = \{X \in \smrg | \ (I,X) \in \cH_{(A_0,A)} \}.
\]
We say that a homogeneous free spectrahedron $\cH_{(A_0,A)}$ is \df{bounded} (i.e. has bounded level sets) if $\hom^{-1} (\cH_{(A_0,A)})$ is bounded. In the case that $A_0$ is positive definite, it is easy to verify that $\hom^{-1} (\cH_{(A_0,A)})$ is the free spectrahedron defined by
\[
\hom^{-1} (\cH_{(A_0,A)}) = \cD_{ A_0^{-1/2} A A_0^{-1/2}}.
\]
Here $A_0^{-1/2}$ is the inverse of the positive definite square root of $A_0$. 

All homogeneous free spectrahedra which arise as the homogenization of some classical free spectrahedron contain the tuple $(1,0, \dots, 0) \in \R^g$ in their interior. We call a homogeneous free spectrahedron $\cH_{(A_0,A)}$  which contains $(1,0,\dots,0)$ in the interior of $\cH_{(A_0,A)} (1)$ a \df{positive homogeneous free spectrahedra}. 

\ssec{Basic properties of homogeneous free spectrahedra}

The following lemma gives several useful properties for homogeneous free spectrahedra. 

\begin{lem}
\label{lem:homspecbasic}
Let $(A_0,A) \in \smdrgp$ and assume $\cH_{(A_0,A)}$ is a positive homogeneous free spectrahedron. Then one has the following
\begin{enumerate}
\item
\label{it:poshomspecA0pos}
The homogeneous component $A_0$ is positive semidefinite. Furthermore, if $(A_0,A)$ is a minimal defining tuple for $\cH_{(A_0,A)}$, then $A_0 \succ 0$.

\item
\label{it:homspecX0pos} If $\cH_{(A_0,A)}$ is bounded, then for any $(X_0,X) \in \cH_{(A_0,A)}$ one has
\[
X_0 \succeq 0 \qquad \mathrm{and} \qquad X_0^{\dagger/2} X X_0^{\dagger/2} \in \hom^{-1} (\cH_{(A_0,A)})
\]
and
\[
\ker X_0 \subset \ker X_i \qquad  \mathrm{for \ all\ } i=1,\dots,g.
\]
Here $X_0^{\dagger/2}$ is the positive semidefinite square root of the Moore-Penrose pseudo inverse of $X_0$.

\item
\label{it:minimalhom}
If $\cH_{(A_0,A)}$ is bounded, then tuple $(A_0,A)$ is a minimal defining tuple for $\cH_{(A_0,A)}$ if and only if $A_0 \succ 0$ and $A_0^{-1/2}A A_0^{-1/2}$ is a minimal defining tuple for $\hom^{-1} (\cH_{(A_0,A)})$.

\item
\label{it:PosSpecIsHomogenization}
A bounded homogeneous free spectrahedron is positive if and only if it can be expressed as the homogenization of a bounded free spectrahedron.

\end{enumerate}
\end{lem}

\begin{proof}
To prove item \eqref{it:poshomspecA0pos} note that by assumption we have $(1,0,\cdots 0) \in \cH_{(A_0,A)}$ hence
\[
\Lambda_{(A_0,A)} (1,0,\dots, 0) = A_0 \succeq 0. 
\]
It remains to show that $A_0 \succ 0$ if $(A_0,A)$ is a minimal defining tuple for $\cH_{(A_0,A)}$. 

To this end let $V$ be an isometry mapping from the kernel of $A_0$ into $\R^d$. By assumption we have $(1,0,\dots, 0)$ in the interior of $\cH_{(A_0,A)}$, so there exists some $\epsilon>0$ such that for all $X \in \R^g$ with $\|X\|_2 < \epsilon$ we have $(1,X) \in \cH_{(A_0,A)}$. Thus we have
\[
V^T \Lambda_{(A_0,A)} (1,X) V = \Lambda_{V^T A V} (X) \succeq 0 \qquad \mathrm{for \ all \ } X \mathrm{\ with \ } \|X\|_2 < \epsilon
\]
from which it follows that $\Lambda_{V^T A V}$ is identically zero hence $V^T A V=0.$ This implies that there is a unitary $U \in \R^d$ such that
\[
U^T (A_0,A) U= (\check{A_0},\check{A}) \oplus (0,0)
\]
where $\check{A_0} \succ 0$. We conclude that if $(A_0,A)$ is a minimal defining tuple then $A_0 \succ 0$. 

We now prove item \eqref{it:homspecX0pos}. By assumption $\cH_{(A_0,A)}$ is positive so using Lemma \ref{lem:homspectuplequiv} with item \eqref{it:poshomspecA0pos} we can reduce to the case $A_0=I$ so that 
\[
\hom^{-1} (\cH{(A_0,A)}) = \cD_A.
\]
Assume towards a contradiction that there is some $(X_0,X) \in \cH_{(I,A)}(n)$ with $X_0 \not \succeq 0$. Then there is an isometry $V: \R \to \R^n$ such that $V^T X_0 V < 0$. But then we have
\[
(I \otimes V^T) \Lambda_{(I,A)} (X_0 , X ) (I \otimes V) = \Lambda_{(I,A)} (V^T X_0 V, V^T X V) = I \otimes (V^T X_0 V) + \Lambda_A (V^T X V) \succeq 0.
\]
If follows that 
\[
\Lambda_A (\alpha V^T X V) \succ 0  \qquad \mathrm{ for \ all  } \qquad \alpha>0,
\]
from which we conclude $\alpha V^T X V \in \cD_A$ for all $\alpha>0$, contradicting the assumption that $\cD_A = \hom^{-1}  (\cH{(I,A)})$ is bounded. Having shown that $X_0$ is positive semidefinite, it is easy to verify that 
\[
 X_0^{\dagger/2} X X_0^{\dagger/2} \in \hom^{-1} (\cH{(A_0,A)}).
\]

The claim 
\[
\ker X_0 \subset \ker X_i \qquad  \mathrm{for \ all\ } i=1,\dots,g.
\]
follows from letting $V$ be an isometry mapping the kernel of $X_0$ into $\R^n$ and repeating the argument above. 

For item \eqref{it:minimalhom}, first assume that $(A_0,A)$ is a minimal defining tuple for $\cH_{(A_0,A)}$. Item \eqref{it:poshomspecA0pos} shows that $A_0$ is positive definite so we need only show that $A_0^{-1/2} A A_0^{-1/2}$ is a minimal defining tuple for $\hom^{-1} (\cH_{(A_0,A)})=\cD_{A_0^{-1/2} A A_0^{-1/2}}.$ Assume towards a contradiction that there is some proper subpencil $A^1$ of $A_0^{-1/2} A A_0^{-1/2}$ such that $\cD_{A_0^{-1/2} A A_0^{-1/2}}=\cD_{A^1}$. Equivalently we have
\beq
\label{eq:equivpencils}
L_{A_0^{-1/2} A A_0^{-1/2}} (X) \succeq 0 \qquad \mathrm{if\ and \ only \ if} \qquad L_{A^1} (X) \succeq 0. 
\eeq
We will show 
\[
\cH_{(I,A_0^{-1/2} A A_0^{-1/2})} = \cH_{(I,A^1)} 
\]

Note that $\cH_{(I,A_0^{-1/2} A A_0^{-1/2})} = \cH_{(A_0,A)}$ is bounded by assumption. Additionally, $\cH_{(I,A^1)}$ is bounded since $\hom^{-1} (\cH_{(I,A^1)}) = \cD_{A^1}$ is bounded. Using item \eqref{it:homspecX0pos} shows that if $(Y_0,Y)$ is an element of either $ \cH_{(A_0,A)}$ or $\cH_{(I,A^1)}$, then up to unitary equivalence we have
\[
(Y_0,Y)=  (\check{Y}_0,\check{Y}) \oplus (0,0)
\]
where $\check{Y}_0 \succ 0$. Furthermore, for $\check{Y}_0 \succ 0$, one has
\[
\Lambda_{(I,A_0^{-1/2} A A_0^{-1/2})} (\check{Y}_0,\check{Y}) \succeq 0 \qquad \mathrm{if\ and \ only \ if} \qquad L_{(A_0^{-1/2} A A_0^{-1/2})} (\check{Y}_0^{-1/2} \check{Y} \check{Y}_0^{-1/2}) \succ 0,
\]
and similarly
\[
\Lambda_{(I,A^1)} (\check{Y}_0,\check{Y}) \succeq 0 \qquad \mathrm{if\ and \ only \ if} \qquad L_{A^1} (\check{Y}_0^{-1/2} \check{Y} \check{Y}_0^{-1/2}) \succ 0.
\]
Combining the above with equation \eqref{eq:equivpencils} shows $\cH_{(A_0,A)}=\cH_{(I,A_0^{-1/2} A A_0^{-1/2})} = \cH_{(I,A^1)}$, hence $(A_0,A)$ is not a minimal defining tuple for $\cH_{(A_0,A)}$. 

It remains to prove the reverse direction. We now assume that $A_0 \succ 0$ and show that if $(A_0,A)$ is not a minimal defining tuple for $\cH_{(A_0,A)}$, then $A_0^{-1/2} A A_0^{-1/2}$ is not a minimal defining tuple for $\cD_{A_0^{-1/2} A A_0^{-1/2}} = \hom^{-1} (\cH_{(A_0,A)})$. By definition, there exists some invertible matrix $V \in M_d (\R)^g$ such that 
\[
V^T (A_0, A)V = (A_0^1,A^1) \oplus (A_0^2,A^2)
\]
and such that $\cH_{(A_0,A)} = \cH_{(A_0^1,A^1)}$. Since $A_0$ is positive definite and $V$ is invertible, we have $A_0^1 \succ 0$, hence
\[
\cD_{A_0^{-1/2} A A_0^{-1/2}} = \hom^{-1} (\cH_{(A_0,A)}) =  \hom^{-1}  (\cH_{(A_0^1,A^1)}) = \cD_{(A_0^1)^{-1/2} A^1 (A_0^1)^{-1/2}}.
\]
The size of the matrices in the tuple $(A_0^1)^{-1/2} A^1 (A_0^1)^{-1/2}$ is strictly smaller than that of the matrices in the tuple $A_0^{-1/2} A A_0^{-1/2}$ so an application \cite[Theorem 3.12]{HKM13} shows that $A_0^{-1/2} A A_0^{-1/2}$ is not a minimal defining tuple for $\cD_{A_0^{-1/2} A A_0^{-1/2}}$.

Item \eqref{it:PosSpecIsHomogenization} is a straightforward consequence of item \eqref{it:minimalhom}.
\end{proof}

\begin{remark}
\label{rem:BadHomogenization}
The assumption that $\cH_{(A_0,A)}$ is bounded in item \eqref{it:minimalhom} of Lemma \ref{lem:homspecbasic} is necessary. In the case of unbounded free spectrahedra, it is possible to have tuples which define the same free spectrahedron, but not the same homogeneous free spectrahedron. An example in a single variable is as follows. Set
\[
A = 1 \qquad \mathrm{and} \qquad
B = \begin{pmatrix} 1 & 0 \\
0 & 1/2
\end{pmatrix}.
\]
Then $\cD_A = \cD_B$ while $\cH_{(I,A)} \neq \cH_{(I,B)}$ as seen from the fact that $(-1,1) \in \cH_{(I,A)}$ but $(-1,1) \notin \cH_{(I,B)}$. Furthermore, $(I,B)$ is a minimal defining tuple for $\cH_{(I,B)}$ while $B$ is not a minimal defining tuple for $\cD_B$. 
\end{remark}

\begin{remark}

\cite[Theorem 3.12]{HKM13} (see also \cite{Zal17}) shows that any two minimal defining tuples for a free spectrahedron $\cD_A$ are unitarily equivalent, hence one may alternatively define a minimal defining tuple for $\cD_A$ as a defining tuple for $\cD_A$ which has smallest size. Using Lemma \ref{lem:homspecbasic} \eqref{it:minimalhom}, one may also define a minimal defining tuple for a bounded positive homogeneous free spectrahedron in this manner. That is, $(A_0,A) \in \smdrgp$ is a minimal defining tuple for $\cH_{(A_0,A)}$ if and only if for any integer $d_1$ and any tuple $(B_0,B) \in SM_{d_1} (\R)^{g+1}$ such that $\cH_{(A_0,A)}=\cH_{(B_0,B)}$ one has $d \leq d_1$. 

\end{remark}

\subsection{The Arveson boundary of homogeneous free spectrahedra}

Lemma \ref{lem:homspecbasic} allows the Arveson boundary of a bounded positive homogeneous free spectrahedron $\cH_{(A_0,A)}$, denoted $\arv(\cH_{(A_0,A)}),$ to be defined in the following manor. For a $(g+1)$-tuple $(X_0,X)\in \smnrgp$ say $(X_0,X)$ is in the \df{Arveson boundary} of bounded positive homogeneous free spectrahedron $\cH_{(A_0,A)}$ if and only if $(X_0,X)$ is unitarily equivalent to a tuple of the form $(Y_0,Y) \oplus (0,0)$ where 
\[
Y_0 \succ 0 \qquad \mathrm{and} \qquad Y_0^{-1/2} Y Y_0^{-1/2} \in \arv(\hom^{-1} (\cH_{(A_0,A)})).
\]
The following theorem is an extension of  Theorem \ref{thm:ArvImpliesEuc} \eqref{it:KerCondition} to the setting of homogeneous free spectrahedra. 

\begin{thm}
\label{thm:HomogenousArvesonKernelCondition}
Let $\cH_{(A_0,A)}$ be a bounded positive homogeneous free spectrahedron with minimal defining tuple $(A_0,A)$. Let $(X_0,X) \in \smnrgp$ with $X_0 \succ 0$ and set $Y=X_0^{-1/2}X X_0^{-1/2}$. Then $(X_0,X)\in \arv \cH_{(A_0,A)}$ if and only if for each non negative integer $m \leq n$ the only $\beta \in M_n (\R)^g$ satisfying
\beq
\label{eq:ArvesonKernel}
\ker \Lambda_{(A_0,A)} (I,Y) \subseteq \ker \Lambda_{(A_0,A)} (I_m \oplus 0_{n-m},\beta)
\eeq
is  \beq
\label{eq:betaEquals}
\beta=
\begin{pmatrix}
Y_{11} & Y_{12} \\
0 & 0_{n-m}
\end{pmatrix} 
\eeq where $Y$ and $\beta$ are written with respect to the block decomposition of $I_m \oplus 0_{n-m}$.
\end{thm}

\begin{proof}
By assumption $(A_0,A)$ is a minimal defining tuple for $\cH_{(A_0,A)}$, so Lemma \ref{lem:homspecbasic} shows that $A_0 \succ 0$. It is straight forward to check that a tuple $(I_m \oplus 0_{n-m},\beta)$ satisfies equation \eqref{eq:ArvesonKernel} if and only if
\[
\ker \Lambda_{(I,A_0^{-1/2} A A_0^{-1/2})} (I,Y) \subseteq \ker \Lambda_{(I,A_0^{-1/2} A A_0^{-1/2})} (I_m \oplus 0_{n-m},\beta)
\]
so we may without loss of generality reduce to the case $A_0=I$. 

Now, to prove the forward direction of the theorem, note that by assumption the only $\beta \in(\R^{n \times n})^g$ satisfying 
\[
\ker \Lambda_{(I,A)} (I,Y)=\subseteq \ker \Lambda_{(I,A)} ( 0_n,\beta)
\]
is $\beta=0$. It follows from Theorem \ref{thm:ArvImpliesEuc} \eqref{it:KerCondition} that $Y \in \arv(\cD_A)$ and therefore $(X_0,X) \in \arv\cH_{(I,A)}$. 


We now prove the reverse direction. Observe that for any $\beta \in M_n (\R)^g$ one has
\[
\ker \Lambda_{(I,A)} \left(
\begin{pmatrix}
I_m & 0 \\
0 & 0_{n-m}
\end{pmatrix},
\begin{pmatrix}
\beta_{11} & \beta_{12} \\
\beta_{21} & \beta_{22}
\end{pmatrix}
\right)
\subseteq 
\ker \Lambda_{(I,A)} \left(
\begin{pmatrix}
0_m & 0 \\
0 & 0_{n-m}
\end{pmatrix},
\begin{pmatrix}
0_m & 0 \\
\beta_{21} & \beta_{22}
\end{pmatrix}
\right)
\]
It follows that if $\beta$ satisfies equation \eqref{eq:ArvesonKernel}, then 
\[
\ker \Lambda_{(I,A)} (I,Y)=\ker L_A (Y) \subseteq \ker \Lambda_A \left(\begin{pmatrix}
0_m & 0 \\
\beta_{21} & \beta_{22}
\end{pmatrix}\right)
\]
 By definition, $(X_0,X) \in \arv(\cH_A)$ if and only if $Y \in \arv (\cD_A)$ so again using Theorem \ref{thm:ArvImpliesEuc} \eqref{it:KerCondition} shows that that $\beta_{21}=0$ and $\beta_{22}=0$. 
 
To complete the proof observe if $\beta$ satisfies equation \eqref{eq:ArvesonKernel} then we have
\[
 \ker \Lambda_{(I,A)} (I,Y) \subseteq \ker \Lambda_{(I,A)} (I,Y)- \Lambda_{(I,A)} (I_m \oplus 0_{n_m},\beta).
\]
Using $\beta_{21}=0$ and $\beta_{22}=0$ gives
\[
\begin{array}{rcl}
\ker \Lambda_{(I,A)} (I,Y)- \Lambda_{(I,A)} (I_m \oplus 0_{n_m},\beta)&=&\ker \Lambda_{(I,A)} \left(
\begin{pmatrix}
0_m & 0 \\
0 & I_{n-m}
\end{pmatrix},
\begin{pmatrix}
Y_{11}-\beta_{11} & Y_{12}-\beta_{12} \\
Y_{21} & Y_{22}
\end{pmatrix}
\right) \\
&\subseteq &
\ker \Lambda_{(I,A)} \left(
\begin{pmatrix}
0_m & 0 \\
0 & 0_{n-m}
\end{pmatrix},
\begin{pmatrix}
Y_{11}-\beta_{11} & Y_{12}-\beta_{12} \\
0 & 0_{n-m}
\end{pmatrix}
\right)
\end{array}
\]
This shows that 
\[\ker L_A (Y) \subseteq \ker \Lambda_A \left(\begin{pmatrix}
Y_{11}-\beta_{11} & Y_{12}-\beta_{12} \\
0 & 0
\end{pmatrix}\right)
\]
As before, it follows from Theorem \ref{thm:ArvImpliesEuc} \eqref{it:KerCondition} that $Y_{11}=\beta_{11}$ and $Y_{12}=\beta_{12}$. 
\end{proof}

\section{Linear and projective transformations of Homogeneous Free Spectrahedra}
\label{sec:LinearAndProjectiveTransformations}

We now discuss linear and projective transformations of free and homogeneous free spectrahedra. The main goal in this section is to show that the Arveson boundary of a free spectrahedron is preserved under invertible projective transformations of the free spectrahedron.

Given a matrix $W \in M_{g} (\R)$ we define a linear transformation $\cT_W$ on $\smrg$ by
\[
\cT_W X=\Big(\sum_{j=1}^g W_{1j} X_j , \dots,\sum_{j=1}^g W_{gj} X_j \Big) \qquad \mathrm{for \ all\ } X \in \smrg.
\]
Note that if we consider $x=(x_1,\dots, x_g)$ as a vector of noncommuting indeterminants, then $\cT_W X$ is the evaluation of $W x$ on the tuple $X$. A linear transformation of a homogeneous free spectrahedron is the set defined by
\[
\cT_W (\cH_{(A_0,A)}) = \{ \cT_W (X) | \ X \in \cH_{(A_0,A)} \}.
\]

Suppose we are given a free spectrahedron $\cD_A$ and a matrix $W \in M_{g+1} (\R)$ such that $\cT_W (\hom (\cD_A))$ is a positive homogeneous free spectrahedron. Then we call $\cT_W$ a \df{positive linear transformation} of $\hom (\cD_A)$ and we define the \df{projective transformation} $\cP_W$ of $\cD_A$ by 
\[
\cP_W (\cD_A) = \hom^{-1}(\cT_W (\hom (\cD_A))).
\] 

We extend the definition of $\hom$ and $\hom^{-1}$ and $\cP_W$ to matrix tuples as follows. Given a matrix tuple $X \in \smrg$, we define
\[
\hom(X)=(I,X) \in \smrgp.
\]
Additionally, given a matrix tuple $(X_0,X) \in \smrgp$ with $X_0 \succeq 0$, we define 
\[
\hom^{-1} (X) = X_0^{\dagger/2} X X_0^{\dagger/2} \in \smrg.
\]

Suppose as before that the free spectrahedron $\cD_A$ and matrix $W \in M_{g+1}(\R)$ are chosen so that $\cT_W (\hom (\cD_A))$ is a positive homogeneous free spectrahedron. Then for $X \in \cD_A$ we define
\[
\cP_W (X)= \hom^{-1} (\cT_W ( \hom(X))).
\]
Recall that the homogeneous component of $\cT_W ( \hom(X))$ is positive semidefinite as a consequence of Lemma \ref{lem:homspecbasic}, hence $\cP_W(X)$ is well defined. With these definitions it is straight forward to check that
\[
\cP_W (\cD_A) = \{\cP_W (X) | \ X \in \cD_A \}.
\] 
In addition, one may easily verify that a projective or linear transformation is invertible if and only if the matrix $W$ is invertible, with inverses given by $\cP_{W^{-1}}$ and $\cT_{W^{-1}}$, respectively.

\subsection{Basic properties of linear and projective transformations}

We now establish several basic properties for linear and projective mappings of (homogeneous) free spectrahedra. Our first lemma shows that a linear transformation of homogeneous free spectrahedron is again a homogeneous free spectrahedron.
\begin{lem}
\label{lem:HomSpecTransformProps}
For any $A,X \in \smrg$ and $V \in M_{g} (\R)$ one has
\beq
\label{eq:LambdaTransformEq}
\Lambda_A (\cT_V (X)) = \Lambda_{\cT_{V^T} (A)} (X)
\eeq
As consequence, if $W \in M_{g+1} (\R)$ is invertible then one has
\beq
\label{eq:HomSpecTransform}
 \cT_W (\cH_{(I,A)})= \cH_{\cT_{W^{-T}} (I,A)} .
\eeq
\end{lem}
\begin{proof}
It is a routine calculation to verify that equation \eqref{eq:LambdaTransformEq} holds for any $A,X \in \smrg$ and any $V \in M_{g} (\R)$ and in the case that $W \in M_n (\R)^g$ is invertible, it is straightforward to check that $\cT_W$ is invertible with inverse $\cT_{W^{-1}}$. Combining these facts gives  $ \cT_W (\cH_{(I,A)})= \cH_{\cT_{W^{-T}} (I,A)} $.
\end{proof}

We next show that linear transformations are free maps in that they respect direct sums and left and right element wise matrix multiplications. An important consequence of this fact is that invertible linear transformations of minimal defining tuples again give minimal defining tuples. 

\begin{lem}
\label{lem:HomMinDefiningTransform}
Linear transformations of matrix tuples respect direct sums and left and right matrix multiplication. That is, for any $A,B \in \smrgp$ and $W \in M_{g+1} (\R)$ one has $\cT_W (A \oplus B)= \cT_W (A) \oplus \cT_W (B)$. If in addition $A$ is a $g$-tuple of $d \times d$ matrices $U \in M_{n \times d} (\R)$ and $V \in M_{d \times m} (\R) $, then 
\[
\cT_W (UAV) = U \cT_W (A) V \in M_{n \times m} (\R)^g.
\]
As a consequence, if $\cT_W$ is invertible and $(A_0,A) \in \smdrgp$ is a minimal defining tuple for $\cH_{(A_0,A)}$, then $\cT_{W^{-T}} (A_0,A)$ is a minimal defining tuple for $\cT_W (\cH_{(A_0,A)})$.
\end{lem}
\begin{proof}
It is a straightforward computation to check that linear transformations respect direct sums and left and right matrix multiplications. To prove the last claim, assume that $W \in M_{g+1} (\R)$ is invertible and that $(A_0,A)$ is not a minimal defining tuple for $\cH_{(A_0,A)}$. We will show that $\cT_{W^{-T}} (A_0,A)$ is not a minimal defining tuple for $\cT_W (\cH_{(A_0,A)})=\cH_{\cT_{W^{-T}} (A_0,A)}$. 

Since $(A_0,A)$ is not a minimal defining tuple for $\cH_{(A_0,A)}$, there is an invertible matrix $V \in M_d (\R)$ such that $V^T (A_0,A) V = (A_0^1,A^1) \oplus (A_0^2,A^2)$ where 
\[
\Lambda_{(A_0^1,A^1)} (X_0,X) \succeq 0 \qquad \mathrm{implies} \qquad \Lambda_{(A_0^2,A^2)} (X_0,X) \succeq 0.
\]
for any $(X_0,X) \in \smrgp$. Using the fact that linear transformations respect direct sums and left and right matrix multiplication we find that $V^T \cT_{W^{-T}} (A_0,A)V$ is equal to $\cT_{W^{-T}} (A_0^1,A^1) \oplus \cT_{W^{-T}} (A_0^2,A^2)$, and an application of Lemma \ref{lem:HomSpecTransformProps} shows that
\[
\Lambda_{\cT_{W^{-T}}(A_0^1,A^1)} \cT_W (X_0,X) \succeq 0 \qquad \mathrm{implies} \qquad \Lambda_{\cT_{W^{-T}}(A_0^2,A^2)} \cT_W (X_0,X) \succeq 0.
\]
We conclude that $\cH_{\cT_{W^{-T}}(A_0^1,A^1)} = \cH_{\cT_{W^{-T}}(A_0,A)}$, hence $\cT_{W^{-T}}(A_0,A)$ is not a minimal defining tuple.
\end{proof}

Using Lemma \ref{lem:HomMinDefiningTransform} allows us to obtain an analogue of Lemma \ref{lem:HomSpecTransformProps} for projective transformations.

\begin{lem}
Let $A \in \smdrg$ be a minimal defining tuple for the free spectrahedron $\cD_A$ and let $W \in M_{g+1} (\R)$ be an invertible matrix such that $\cT_W$ is a positive linear  transformation of the homogeneous free spectrahedron $\cH_{(I,A)}$. Then
\beq
\label{eq:ProjSpecTransform}
\cP_W (\cD_{(A)})= \cD_{\cP_{W^{-T}} (A)}.
\eeq
In particular, an invertible projective transformation of a free spectrahedron is again a free spectrahedron
\end{lem}
\begin{proof}
By definition we have $\hom (\cD_A) = \cD_{\hom (A)}$ hence an application of Lemma \ref{lem:HomSpecTransformProps} gives
\[
\cP_W (\cD_A) = \hom^{-1} (\cH_{\cT_{W^{-T}} (\hom (A))}).
\]
Furthermore, Lemma \ref{lem:HomMinDefiningTransform} shows that $\cT_{W^{-T}} (\hom (A))$ is a minimal defining tuple for $\cH_{\cT_{W^{-T}} (\hom (A))}$, hence we may use Lemma \ref{lem:homspecbasic} item \eqref{it:poshomspecA0pos} to conclude that the homogeneous component of this matrix tuple is positive definite. It follows that $\hom^{-1} (\cT_{W^{-T}} (\hom (A)))$ is well defined and that
\[
\cP_W (\cD_A) =\cD_{\hom^{-1}(\cT_{W^{-T}} (\hom (A)))} = \cD_{\cP_{W^{-T}} (A)}.
\]
\end{proof}

\subsubsection{Boundedness of free spectrahedra under projective mappings}

It is important to note that an invertible projective transformation of a free spectrahedron does not necessarily result in a bounded free spectrahedron.

\begin{example}
\label{exa:UnboundedImage}
Let $\cC$ be the free square. That is $\cC$ is the free spectrahedron with defining pencil $L_C (X)$ where
\[
C = \left( \begin{pmatrix} 1 & 0 & 0 & 0 \\
0 & 0 & 0 & 0 \\
0 & 0 & -1 & 0 \\
0 & 0 & 0 & 0
\end{pmatrix} ,
\begin{pmatrix} 0 & 0 & 0 & 0 \\
0 & 1 & 0 & 0 \\
0 & 0 & 0 & 0 \\
0 & 0 & 0 & -1
\end{pmatrix} 
\right).
\]
In addition let $W \in M_{3} (\R)$ be the matrix
\[
W=\begin{pmatrix} 0 & 1 & -1 \\
1 & -4 & 3  \\
-1 & 5 & -3
\end{pmatrix}
\]
Then $W$ is invertible and $\cT_W( \hom(\cC))$ is a positive homogeneous free spectrahedron, so $\cP_W (\cC)$ is well defined. However, for $-1 \leq y < 1$ we have 
\[
(1,y) \in \cC (1) \qquad \mathrm{and} \qquad \cP_W (1,y)= \left(-3,3+\frac{1}{1-y}\right) 
\]
which shows $\cP_W (\cC)$ is not bounded.
\end{example}

The following lemma gives various conditions related to boundedness of the image of a free spectrahedron under a projective map.

\begin{lem}
\label{lem:BoundedNecessaryConditions}
Let $\cD_A \subset \smrg$ be a free spectrahedron and let $\cP_W$ be an invertible projective transformation of $\cD_A$. Write $W =(w_{i,j}) \in M_{g+1} (\R)$. Then we have the following.
\begin{enumerate}
\item \label{it:InvertibleSubMatrix} If $\cD_A$ is bounded, then the $g \times g$ submatrix of $W$ which is obtained by deleting the first row and column of $W$ is invertible. 
\item \label{it:BoundedNecessary} If $\cP_W (\cD_A)$ is bounded, then $w_{1,1} >0$.
\item \label{it:BoundedSufficient} If $\cD_A$ is bounded and $w_{1,1}>0$ and $w_{1,j}=0$ for all $j>0$, then $\cP_W (\cD_A)$ is bounded. 
\end{enumerate} 
\end{lem}

\begin{proof}
We first prove item \eqref{it:BoundedNecessary}. Considering the image of $(1,0,\dots,0) \in \R^{g+1}$ under $\cT_W$ shows that $(w_{1,1}, w_{1,2}, \dots, w_{1,g+1}) \in \cT_W (\hom (\cD_A))$. Furthermore, from our definitions of projective mappings and boundedness for homogeneous free spectrahedra, we know that $\cT_W (\hom (\cD_A))$ is a bounded positive homogeneous free spectrahedron. Since $\cT_W$ is invertible we have $(w_{1,1}, w_{1,2}, \dots, w_{1,g+1}) \neq 0$, hence an application of Lemma \ref{lem:homspecbasic} \eqref{it:homspecX0pos} shows $w_{1,1} >0$.

We now prove item \eqref{it:InvertibleSubMatrix}. In this case $\cP_{W^{-1}}$ is a projective transformation of $\cP_W (\cD_A)$ to $\cD_A$, so using item \eqref{it:BoundedNecessary} shows that the $1,1$ entry of $W^{-1}$ is positive. The invertibility of the desired submatrix of $W$ follows from an application of Cramer's rule. 

To prove \eqref{it:BoundedSufficient} observe that if $w_{1,j}=0$ for each $j>1$, then $\cP_W$ acts on as an invertible affine transformation on $\R^g$, hence $\cP_W$ maps bounded subsets of $\R^g$ to bounded subsets of $\R^g$. It follows that level one of $\cP_W (\cD_A)$ is bounded, hence $\cP_W (\cD_A)$ is bounded. 
\end{proof}

\subsection{Projective transformations and the Arveson boundary}

We now establish the main result of the section. Namely, we show that invertible linear and projective transformations of free spectrahedra map Arveson boundary to Arveson boundary. 

\begin{theorem}
\label{thm:ArvMappings}
Let $A \in \smdrg$ and let $W \in M_{g+1} (\R)$ be a matrix such that $\cP_W$ is an invertible projective transformation of $\cD_{A}$. Assume that both $\cD_A$ and $\cP_W (\cD_A)$ are bounded free spectrahedra. Then for any $X \in \smrg$ one has $X \in \arv \cD_A$ if and only if $\cP_W (X) \in \arv (\cP_W (\cD_A))$. Equivalently one has
\[
\arv (\cP_W (\cD_A))=\cP_W (\arv \cD_A)).
\]
\end{theorem}

\begin{proof}
As previously mentioned, $\cP_W$ is invertible if and only if the matrix $W$ is invertible and the inverse of $\cP_W$ is $\cP_{W^{-1}}$. Therefore, it is sufficient to show that if $X \notin \arv \cD_A $, then $\cP_W (X) \notin \arv \cP_W( \cD_A)$. In particular we show that $(I,X) \notin  \arv \cH_{(I,A)}$ implies $\cT_W (I,X) \notin \arv  \cT_W(\cH_{(I,A)})$. Furthermore, we may WLOG assume that $A$ is a minimal defining tuple for $\cD_A$, hence $(I,A)$ is a minimal defining tuple for $\cH_{(I,A)}$.

Assume $(I,X) \in \cH_{(I,A)} \backslash \arv \cH_{(I,A)}$. As a consequence of \cite[Lemma 2.1]{EH19}, there must exist some nonzero tuple $\beta \in M_n (\R)^g$ such that
\[
\ker \Lambda_{(I,A)} (I,X) \subseteq \ker \Lambda_{(I,A)} (0,\beta).
\]
Using Lemma \ref{lem:HomSpecTransformProps} we obtain
\beq
\label{eq:KernelTransformOne}
\ker \Lambda_{\cT_{W^{-T}} (I,A)} \cT_W (I,X) \subset \ker \Lambda_{\cT_{W^{-T}} (I,A)}  \cT_W (0,\beta).
\eeq
Set $(\tY_0,\tY):= \cT_W (I,X)$ and $(\tB_0,\tB)=\cT_{W^{-T}} (I,A)$. We first show that  $\tB_0 \succ 0$ and $\tY_0 \succ 0$

From our definition of a projective map, the homogeneous free spectrahedron $\cH_{\cT_{W^{-T}} (I,A)} $ is positive, and Lemma \ref{lem:HomMinDefiningTransform} shows that ${\cT_{W^{-T}} (I,A)}$ is a minimal defining tuple. It follows from Lemma \ref{lem:homspecbasic} \eqref{it:poshomspecA0pos} that $\tB_0 \succ 0$. Furthermore, using Lemma \ref{lem:HomMinDefiningTransform} together with Lemma \ref{lem:homspecbasic} \eqref{it:homspecX0pos} allows us to conclude that if $\tY_0$ has a nontrivial null space, then each component of $\cT_{W^{-1}} (\tY_0,\tY)=(I,X)$ must have a nontrivial null space, which not possible. Therefore we must have $\tY_0 \succ 0$. 

Having shown $\tY_0 \succ 0$ and $\tB_0 \succ 0$, we introduce the following notation:
\[
(I,Y):= \tY_0^{-1/2} (\tY_0,\tY) \tY_0^{-1/2}, \qquad (\gamma_0,\gamma):= \tY_0^{-1/2} \cT_W (0,\beta) \tY_0^{-1/2} , \qquad (I,B):= \tB_0^{-1/2} (\tB_0,\tB) \tB_0^{-1/2}.
\]
With this notation, equation \eqref{eq:KernelTransformOne} is equivalent to
\beq
\label{eq:KernelTransformTwo}
\ker \Lambda_{(I,B)} (I,Y) \subseteq \ker \Lambda_{(I,B)} (\gamma_0,\gamma).
\eeq
Note that $(\gamma_0,\gamma) \neq 0$ since $\cT_W$ is invertible and $(0,\beta) \neq 0$ by assumption. Therefore if $\gamma_0 =0$, then $\gamma \neq 0$. In this case we may use Theorem \ref{thm:ArvImpliesEuc} \eqref{it:KerCondition} together with equation \eqref{eq:KernelTransformTwo} to conclude $(I,Y) \notin \arv \cT_W (\cH_{(A_0,A)})$, hence $\cT_W (I,X) \notin \arv \cT_W (\cH_{(A_0,A)})$.

Now assume $\gamma_0 \neq 0$. In this case, equation \eqref{eq:KernelTransformTwo} implies
\beq
\label{eq:KernelTransformThree}
\ker \Lambda_{(I,B)} (I,Y) \subseteq \ker (I \otimes \gamma_0^\dagger) \Lambda_{(I, B)} (\gamma_0,\gamma)=\ker \Lambda_{(I, B)} ( \gamma_0^\dagger \gamma_0,\gamma_0^\dagger \gamma).
\eeq
Note that there exists a unitary $U$ such that $U^T\gamma_0^\dagger \gamma_0 U=I_m \oplus 0_{n-m}$ for some positive integer $m \leq n$. Left multiplying equation \eqref{eq:KernelTransformThree} by $I \otimes U^T$ and right multiplying by $I \otimes U$. gives
\beq
\label{eq:KernelTransformFour}
\ker \Lambda_{(I,B)} (I,U^T YU) \subseteq \ker \Lambda_{(I, B)} ( I_m \oplus 0_{n-m}, U^T \gamma_0^\dagger \gamma U).
\eeq
Since the Arveson boundary is closed under unitary conjugation, we can without loss of generality assume $U=I$ and that $\gamma^\dagger_0 \gamma_0 = I_m \oplus 0_{n-m}$. 

Write 
\[
 Y =
\begin{pmatrix}
\Psi^{11} & \Psi^{12} \\
\Psi^{21} & \Psi^{22}
\end{pmatrix}
\]
with respect to the block decomposition of $I_m \oplus 0_{n-m}$ and suppose towards a contradiction that $(I,Y) \in \arv(\cH_{(I,B)})$. In this case Theorem \ref{thm:HomogenousArvesonKernelCondition} implies that 
\beq
\label{eq:BadEqualsTransform}
(I_m,\Psi^{11})=(I_m,V^T \gamma_0^\dagger \gamma V)
\eeq 
 where $V:R^m \to \R^n$ is the isometry 
\[
V=\begin{pmatrix}
I_{m \times m} \\
0_{(n-m) \times m}
\end{pmatrix}.
\]

Expanding the right hand side of  \eqref{eq:BadEqualsTransform} gives 
\beq
\label{eq:LongFormBeta}
(I_m,\Psi^{11})=V^T \gamma^\dagger_0 \tilde{Y}_0^{-1/2} \cT_W (0,\beta) \tilde{Y}_0^{-1/2} V
\eeq
Lemma \ref{lem:HomMinDefiningTransform} shows that linear transformations respect left and right matrix multiplication so we obtain
\[
(I_m,\Psi^{11}) = V^T ( \gamma^\dagger_0 \gamma_0, \gamma_0^\dagger \gamma ) V = \cT_W (0,V^T \gamma^\dagger_0 \tilde{Y}_0^{-1/2} \beta \tilde{Y}_0^{-1/2} V)
\]
or equivalently
\beq
\label{eq:BadMapTransform}
\cT_{W^{-1}} (I_m, \Psi^{11})  =(0,V^T \gamma^\dagger_0 \tilde{Y}_0^{-1/2} \beta \tilde{Y}_0^{-1/2} V).
\eeq

To complete the proof note that $(I_m, \Psi^{11}) \in \cT_{W} (\cH_{(A_0,A)})$ since homogeneous free spectrahedra are closed under isometric conjugation, therefore we must have 
\[
(0,V^T \gamma^\dagger_0 \tilde{Y}_0^{-1/2} \beta \tilde{Y}_0^{-1/2} V) \in \cH_{(A_0,A)}.
\]
Using Lemma \ref{lem:homspecbasic} \eqref{it:homspecX0pos}, we conclude that 
\[
(0,V^T \gamma^\dagger_0 \tilde{Y}_0^{-1/2} \beta \tilde{Y}_0^{-1/2} V) = (0,0) \in \smmrgp,
\]
a contradiction to equation \eqref{eq:BadMapTransform}. We conclude that $(I,Y) \notin \arv \cH_{(A_0,A)}$ and moreover $(\tY_0,\tY)=\cT_W (I,X) \notin \arv \cH_{(A_0,A)}$.  
\end{proof}

\subsection{Projective transformations vs. Euclidean extreme points}

We end the section by briefly examining the relationship between projective and Euclidean extreme points. A perhaps surprising fact is that the image of a Euclidean extreme point under a projective transformation is not necessarily a Euclidean extreme point. 

\begin{example}
Let $\cC$ be the free matrix square as defined in example \ref{exa:UnboundedImage}. Set 
\[
X = \left( \begin{pmatrix} \frac{1}{\sqrt{2}} & \frac{1}{\sqrt{2}} \\ \frac{1}{\sqrt{2}} & \frac{-1}{\sqrt{2}}
\end{pmatrix},
\begin{pmatrix}
1 & 0 \\
0 & 0
\end{pmatrix}\right)
\]
The tuple $X$ is a boundary point of $\cC$ but is not a Euclidean extreme point of $\cC$. However, if one sets 
\[
V=  \frac{1}{280}  \begin{pmatrix} 306 & 0 & 54 \\
-54 & 63 & -108 \\
27 & 126 & 27
\end{pmatrix}
\]
then $\cP_V$ is an invertible projective map of $\cC$ and $\cP_V(\cC)$ is a bounded free spectrahedron. Moreover, by checking the appropriate kernel condition described in \cite[Corollary 2.3]{EHKM18} it can be shown that $\cP_V (X)$ is a Euclidean extreme point of $\cP_V (\cC)$.
\end{example}

\section{The Arveson Boundary of a Free Quadrilateral}
\label{sec:QuadrilateralArvesonBoundary}

In this section we give our classification of the Arveson boundary  of free quadrilaterals. We begin by examining the free square. Recall that the free square is the free spectrahedron in two variables defined by
\[
\cC=\{X \in SM(\R)^2: I \succeq X_1^2 \mathrm{\ and \ } I\succeq X_2^2\}.
\]
Equivalently $\cC$ is the free spectrahedron defined by the linear pencil
\[
L_C (x)=I+
\begin{pmatrix}
1 & 0 & 0 & 0 \\
0 & 0 & 0 & 0 \\
0 & 0 & -1 & 0 \\
0 & 0 & 0 & 0 \\
\end{pmatrix}
x_1
+\begin{pmatrix}
0 & 0 & 0 & 0 \\
0 & 1 & 0 & 0 \\
0 & 0 & 0 & 0 \\
0 & 0 & 0 & -1 \\
\end{pmatrix}
x_2.
\]
As shown by \cite[Proposition 7.1]{EHKM18} the Arveson boundary of the free square is given by
\[
\arv \cC=\euc \cC=\{X \in M(\R^2): I - X_1^2 = 0 \mathrm{\ and \ } I-X_2^2 = 0 \}.
\]
That is, the Arveson boundary of the free square is a noncommutative variety.

We shall use projective mappings of the free square to obtain noncommutative polynomials which annihilate Arveson boundary of any given free quadrilateral. It is well known in the classical setting that all quadrilaterals are projectively equivalent, i.e. that they may be mapped to each other via invertible projective transformations. The following lemma extends this result to the free setting.

\begin{lem}
\label{lem:CubeProjMaps}
Let $\cD_A$ and $\cD_B$ be any free quadrilaterals. Then there exists an invertible projective mapping $\cP_W$ such that $\cP_W (\cD_A)=\cD_B$.
\end{lem}

Since a free spectrahedron is not uniquely determined by its first level, we cannot directly apply the classical result to obtain the result in the free setting. Similar to the classical case, we prove the result in free setting by constructing a sequence of projective maps which map the defining tuple of an arbitrary free quadrilateral to a sequence of standard defining tuples. The key issue compared to the classical setting is showing that one may preserve positivity of the homogeneous component in each step. Technical details are given in Appendix \ref{sec:AppProjLemProof}.

We are now in position to prove that the Arveson boundary of a free quadrilateral is determined by a noncommutative variety. 

\bs
\noindent\textit{Proof of Theorem \ref{thm:ArvIsVariety}}.
To begin, let $\cC_H$ denote the homogeneous free square. That is $\cC_H= \hom (\cC)$ where $\cC$ is as defined above. Set
\[
c_1(x_0,x):=x_0-x_1 x_0^{-1} x_1 \quad \quad c_2 (x_0,x):=x_0-x_2 x_0^{-1} x_2.
\]
Then for any tuple $(X_0,X) \in \smrgp$ with $X_0 \succ 0$ one has $(X_0,X) \in \arv \cC_H$ if and only $c_1 (X_0,X)=c_2 (X_0,X)=0.$

Using Lemma \ref{lem:CubeProjMaps} shows that there exists an invertible projective map $\cP_V$ such that $\cP_V(\cD_A)=\cC$. It follows that $\cT_V$ is a positive linear mapping of $\cH_{(I,A)}$ onto $\cC_H$. Theorem \ref{thm:ArvMappings} then implies that if $X_0$ is invertible then $(X_0,X) \in \arv \cH_{(I,A)}$ if and only if $X \in \cH_{(I,A)}$ and the rational functions $\hat{r}_1 (x_0,x)$ and $\hat{r}_2 (x_0,x)$ defined by
\[
\hat{r}_1 (x_0,x):=c_1(\cT_V (x_0,x)) \quad \quad \hat{r}_2(x_0,x):=c_2(\cT_V (x_0,x))
\]
satisfy $\hat{r}_1(X_0,X)=\hat{r}_2(X_0,X)=0$. In particular, we have $(I,X) \in \arv \cH_{(I,A)}$ and more over $X \in \arv \cD_A$ if and only if $X \in \cD_A$ and $r_1(X)=r_2(X)=0$ where $r_1$ and $r_2$ are rational functions defined by
\[
r_1(x):=\hat{r}_1 (1,x) \quad \quad r_2 (x):= \hat{r}_2 (1,x).
\]

Using the noncommutative Gr{\"o}bner basis algorithm found NCAlgebra's NCGBX package, it was computed that $r_1 (x)$ and $r_2 (x)$ generate a noncommutative Gr{\"o}bner basis of the form $\{p_1,p_2,p_3,p_4\}$ where $p_1,p_2,p_3,p_4$ are all polynomials. We conclude that the noncommutative polynomials $p_1,p_2,p_3,p_4$ satisfy $p_i (X)=0$ for $i=1,\dots, 4$ if and only if $r_1 (X)=r_2(X)=0$. It follows that $X \in \arv \cD_A$ if and only if $X \in \cD_A$ and $p_i (X)=0$ for $i=1,\dots, 4$.

As a technical note, to simplify the  Gr{\"o}bner basis computation we make the change of variables $z_1:=v_{21}+v_{22} x_1+ v_{23}x_2$ and $z_2:=v_{31}+v_{32} x_1+ v_{33}x_2$ where the $v_{ij}$ are the entries of the matrix $V$. Using Lemma \ref{lem:BoundedNecessaryConditions} \eqref{it:InvertibleSubMatrix} one may show that there exist constants $\alpha_0,\alpha_1,\alpha_2 \in \R$ such that and
\[
v_{11}+v_{12} x_1+ v_{13}x_2=\alpha_0+\alpha_1 z_1+\alpha_2 z_2,
\]
hence this change of variables is justified. In addition, the invertibility of $V$ guarantees that $\alpha_0 \neq 0$. 
\qed

\begin{remark}
A step by step computation of the noncommutative Gr{\"o}bner basis in the above proof as well as a computation using NCGBX can be found in the online appendix \url{https://github.com/NCAlgebra/UserNCNotebooks/tree/master/Evert/FreeQuadrilaterals}. In addition, the online appendix contains a Mathematica notebook with functions for computing the noncommutative polynomials which determine the Arveson boundary of a given free quadrilateral.
\end{remark}

\subsection{Free extreme points of free quadrilaterals}

We end the article with a brief examination of the free extreme points of free quadrilaterals. In particular we show that a free extreme point of a free quadrilateral must have bounded size.

\begin{prop}
\label{prop:BoundedFreeExtDim}
Let $\cD_A$ be a free quadrilateral. If $X \in SM_n (\R)^g$ is a free extreme point of $\cD_A$, then $n \leq 2$. As a consequence one has that $\cD_A$ is the matrix convex hull of the set of Arveson extreme points of $\cD_A$ which are elements of $\cD_A (2)$. In notation
\[
\cD_A = \comat ((\arv \cD_A) (2)).
\]
\end{prop}
\begin{proof} A consequence of Lemma \ref{lem:HomMinDefiningTransform} is that the image of a reducible tuple under an invertible projective map is again reducible. By combining this fact with Theorem \ref{thm:ArvMappings} and Lemma \ref{lem:CubeProjMaps}, it is sufficient to restrict to the case where $\cD_A=\cC$ is the free matrix square. Using \cite[Theorem 1.1]{EHKM18}, a tuple $X \in \cC$ is a free extreme point of $\cC$ if and only if it is Arveson extreme and irreducible. Therefore it is sufficient to show that if $X$ is an Arveson extreme point of $\cC$ and $n>2$, then $X$ is is reducible.

As previously discussed, the Arveson boundary of the free matrix square is given by 
\[
\arv \cC=\{X \in SM(\R^2): I - X_1^2 = 0  \mathrm{\ and \ } I- X_2^2 = 0\}.
\]
Suppose $X=(X_1,X_2) \in SM_n (\R)^2$ is an Arveson extreme point of $\cC$ where $n \geq 2$. Then all the eigenvalues of each $X_i$ are either $1$ or $-1$. For $i=1,2$ let $\cE_i (\lambda)$ be the eigenspace of $X_i$ corresponding to the eigenvalue $\lambda$. Then we must have
\[
\dim \cE_i (1) +\dim \cE_i (-1)=n \qquad \mathrm{for \ } i =1,2. 
\]
Suppose there is some eigenspace, say $\cE_1 (1)$ with $\dim \cE_1 (1) > n/2$. Then a dimension count shows that either
\[
 \cE_1 (1) \cap \cE_2 (1) \neq \emptyset \qquad   \mathrm{ or } \qquad \cE_1 (1) \cap \cE_2 (-1) \neq \emptyset.
\]
In either case, we find that $X_1$ and $X_2$ have a common eigenvector, hence the tuple $(X_1,X_2)$ is reducible. 

Now suppose that neither $X_1$ nor $X_2$ has an eigenspace with dimension greater than $n/2$. Then $n$ must be even and we have
\[
\dim E_i (1)=\dim E_i (-1) = n/2 \qquad \mathrm{for \ } i=1,2.
\]
In this case there must exist unitaries $U_1,U_2 \in M_n (\R)$ such that $U_i^T (I_{n/2} \oplus -I_{n/2}) U_i$ for each $i=1,2$. Using this fact it is straightforward to show that for each $i=1,2$ the set of symmetric matrices which commute with $X_i$ is a $\frac{n(n+2)}{4}$ dimensional subspace of $SM_n (\R)$. Using $\dim SM_n (\R)=\frac{n(n+1)}{2}$ we conclude that the set of symmetric matrices commuting with both $X_1$ and $X_2$ is a subspace of $SM_n (\R)$ with dimension at least
\[
2 \frac{n(n+2)}{4} -\frac{n(n+1)}{2} = \frac{n}{2}.
\]
It follows that if $n>2$, then there is a symmetric matrix which is not a constant multiple of the identity that commutes with both $X_1$ and $X_2$, hence the tuple $X$ is reducible. We conclude that if $X \in SM_n (\R)^2$ is a free extreme point of $\cC$, then $n \leq 2$. 

The claim that a free quadrilateral is the matrix convex hull of its Arveson extreme points at level $2$ is a consequence of the first part of the proposition together with \cite[Theorem 1.1]{EH19} which shows that any free spectrahedron is the matrix convex hull of its free extreme points. 
\end{proof}

\begin{remark}
Proposition \ref{prop:BoundedFreeExtDim} does not generalize to free spectrahedra in more than two variables. For example one may consider the $g$ variable free cube $\cC^g$ defined by 
\[
\cC^g = \{X \in SM (\R)^g | \ X_i^2 \preceq I \mathrm{\ for \ } i=1,\dots,g\}.
\]
Then  \cite[Proposition 7.1]{EHKM18} shows that $\cC^g$ has Arveson boundary 
\[
\arv \cC^g = \{X \in SM (\R)^g | \ I - X_i^2 = 0 \mathrm{\ for \ } i=1,\dots,g\}.
\]
It is not difficult to show that there exist irreducible tuples in $X \in \arv \cC^g (n)$ for all $n$, hence $\cC^g$ has free extreme points at all levels $n$.
\end{remark}

\section{Appendix}

\subsection{Proof of Lemma \ref{lem:CubeProjMaps}} 
\label{sec:AppProjLemProof}

We now prove Lemma \ref{lem:CubeProjMaps}.
\begin{proof}
Given free quadrilaterals $\cD_A$ and $\cD_B$ we will show that there is an invertible matrix $V \in M_3 (\R)$ such that 
\[
\cT_V (I,A) = (\tB_0,B) \qquad \mathrm{where \ } \tB_0 \succ 0 \mathrm{ \ and \ } \tB_0^{-1/2} \tB \tB_0^{-1/2} = B. 
\]
In this argument we assume that the $A$ and $B$ are minimal defining tuples for the free quadrilaterals $\cD_A$ and $\cD_B$, respectively.

Since the composition of projective maps is again a projective map, it is sufficient to treat the case where $B$ is fixed and $A$ is arbitrary. We set
\[
B_1 =\begin{pmatrix} 1 & 0 & 0 & 0 \\
0 & -1 & 0 & 0 \\
0 & 0 & -1 & 0\\ 
0 & 0 & 0 & 1
\end{pmatrix} \qquad 
B_2 =\begin{pmatrix} 1 & 0 & 0 & 0 \\
0 & 1 & 0 & 0 \\
0 & 0 & -1 & 0\\ 
0 & 0 & 0 & -1
\end{pmatrix}
\]
and for $i=1,2$ we let $A_i$ be a $4 \times 4$ diagonal matrix with diagonal elements $a_{ij}$ for $j=1,2,3,4$. We may rearrange the diagonal elements of the tuple $A$ without changing the free spectrahedron $\cD_A$, so we WLOG assume that the diagonal elements are arranged such that 
\[
\theta (a_{11},a_{21}) < \theta (a_{12},a_{22})  < \theta (a_{13},a_{23})  < \theta (a_{14},a_{24})
\]
where $\theta (a_{1j},a_{2j})$ denotes the angle the tuple $(a_{1j},a_{2j})$ forms with the positive $x$-axis when proceeding counter clockwise. 

It straight forward to check that there is an invertible matrix $U \in M_3 (\R)$ of the form
\[
U =\begin{pmatrix}
1 & 0 & 0 \\
0 & u_{22} & u_{23} \\
0 & u_{32} & u_{33} \\
\end{pmatrix}
\]
such that $\cT_{U} (I,A_1,A_2) = (I,\hB_1, \hB_2)$ where the matrices $\hB_1$ and $\hB_2$ have the form
\[
\hB_1 =\begin{pmatrix} \hb_{11} & 0 & 0 & 0 \\
0 & -\hb_{12} & 0 & 0 \\
0 & 0 & -\hb_{12} & 0\\ 
0 & 0 & 0 & \hb_{14}
\end{pmatrix} \qquad
\hB_1 =\begin{pmatrix} \hb_{21} & 0 & 0 & 0 \\
0 & \hb_{22} & 0 & 0 \\
0 & 0 & -\hb_{23} & 0\\ 
0 & 0 & 0 & -\hb_{23}
\end{pmatrix}
\]
where $\hb_{ij} \geq 0 $ if either $i>1$ or $j>1$. An application of Lemma \ref{lem:BoundedNecessaryConditions} \eqref{it:BoundedSufficient} shows that the homogeneous free spectrahedron $\cH_{(I,\hB_1,\hB_2)} = \cT_{U^{-1}} \cH_{(I,B)}$ is bounded from which it follows $\hb_{12}+\hb_{14}>0$ and $\hb_{22}+\hb_{23} >0$. Intuitively, $\cT_U$ is a linear map which sends the quadrilateral with corners $(a_{1j},a_{2j})$ to a quadrilateral with one side parallel to the $x$-axis and an adjacent side parallel to the $y$-axis.

Now let $W \in M_3 (\R)$ be matrix
\[
W = \begin{pmatrix} 
1 & 0 & 0 \\
\frac{\hb_{12}-\hb_{14}}{\hb_{12}+\hb_{14}} & \frac{2}{\hb_{12}+\hb_{14}} & 0 \\
\frac{\hb_{23}-\hb_{22}}{\hb_{23}+\hb_{22}} & 0 &  \frac{2}{\hb_{23}+\hb_{22}}
\end{pmatrix}
\]
Then $W$ is invertible and we have 
\[
\cT_W (I,\hB_1,\hB_2) = 
\left( \begin{pmatrix} 1 & 0 & 0 & 0 \\
0 & 1 & 0 & 0 \\
0 & 0 & 1 & 0\\ 
0 & 0 & 0 & 1
\end{pmatrix}, 
\begin{pmatrix} \beta_1 & 0 & 0 & 0 \\
0 & -1 & 0 & 0 \\
0 & 0 & -1 & 0\\ 
0 & 0 & 0 & 1
\end{pmatrix},
\begin{pmatrix} \beta_2 & 0 & 0 & 0 \\
0 & 1 & 0 & 0 \\
0 & 0 & -1 & 0\\ 
0 & 0 & 0 & -1
\end{pmatrix}.
\right)
\]

Write $\cT_W(I,\hB_1,\hB_2)= (I,\check{B}_1,\check{B}_2).$ Once again using Lemma \ref{lem:BoundedNecessaryConditions} \eqref{it:BoundedSufficient} shows that the free spectrahedron $\cD_{(\check{B}_1,\check{B}_2)}$ is bounded. Using \cite[Proposition 4.3]{HKMjems} shows that $0$ must be in interior the convex hull of $(\beta_1,\beta_2),(-1,1),(-1,-1),$ and $(1,-1)$ which allows us to conclude that $\beta_1+\beta_2 > 0$ and $\beta_1 >-1$ and $\beta_2>-1$. 

Finally let $Z \in M_3 (\R)$ be the matrix 
\[
Z = \frac{1}{2+\beta_1+\beta_2} \begin{pmatrix} 2+\beta_1 +\beta_2 & 1- \beta_2 & 1-\beta_1 \\
\beta_1-\beta_2 & 1+ 2 \beta_1 +\beta_2 & -1 +\beta_1 \\
\beta_2 - \beta_1 & -1 +\beta_2 & 1+ \beta_1 + 2 \beta_2 
\end{pmatrix}
\]
Then $Z$ has determinant
\[
\det (Z) = \frac{8 (1+\beta_1) (1+\beta_2) (\beta_1+\beta_2)}{(2+\beta_1+\beta_2)^3}
\]
hence the constraints $\beta_1+\beta_2 > 0$ and $\beta_1 >-1$ and $\beta_2>-1$ guarantee that $Z$ is invertible. Setting $\cT_{Z^{-1}} (I,\check{B}_1,\check{B}_2) = (\tB_0,\tB_1,\tB_2)$ we obtain
\[
\tB_0 = 
\begin{pmatrix}
\frac{2+\beta_1+\beta_2}{4} & 0 & 0 & 0 \\
0 & \frac{2+\beta_1+\beta_2}{2+2\beta_2} & 0 & 0 \\
0 & 0 & \frac{2+\beta_1+\beta_2}{2(\beta_1+\beta_2)} & 0 \\
0 & 0 & 0 & \frac{2+\beta_1+\beta_2}{2 + 2\beta_1} 
\end{pmatrix}, \quad 
\tB_1 \begin{pmatrix}
\frac{2+\beta_1+\beta_2}{4} & 0 & 0 & 0 \\
0 & -\frac{2+\beta_1+\beta_2}{2+2\beta_2} & 0 & 0 \\
0 & 0 & -\frac{2+\beta_1+\beta_2}{2(\beta_1+\beta_2)} & 0 \\
0 & 0 & 0 & \frac{2+\beta_1+\beta_2}{2 + 2\beta_1} 
\end{pmatrix}
\]
and
\[
\tB_2= \begin{pmatrix}
\frac{2+\beta_1+\beta_2}{4} & 0 & 0 & 0 \\
0 & \frac{2+\beta_1+\beta_2}{2+2\beta_2} & 0 & 0 \\
0 & 0 & -\frac{2+\beta_1+\beta_2}{2(\beta_1+\beta_2)} & 0 \\
0 & 0 & 0 & -\frac{2+\beta_1+\beta_2}{2 + 2\beta_1} 
\end{pmatrix}
\]
The constraints on $\beta_1$ and $\beta_2$ guarantee that $\tB_0$ is positive definite and by construction we have $B_i = \tB_0^{-1/2} \tB_i \tB_0^{-1/2}$ for $i=1,2$.   Setting $V = Z^{-1} W U$ gives an invertible matrix such that $\cP_V (A)=B$. Using lemma \ref{lem:HomSpecTransformProps} we conclude that $\cP_{V^{-T}} (\cD_A) =\cD_B $.
\end{proof}

\newpage
\centerline{NOT FOR PUBLICATION}

\tableofcontents


\begin{thebibliography}{DLTW08}


\bibitem[A88]{A88} J. Agler, {\it An abstract approach to model theory}, Surveys of some recent results in operator theory, Vol. II, 1-23, Pitman Res. Notes Math. Ser., 192, Longman Sci. Tech., Harlow, 1988.

\bibitem[AM15]{AM15}
J. Agler, J.E. McCarthy:
\textit{Global holomorphic functions in several non-commuting variables},
{Canad. J. Math.} {67} (2015) 241--285.


\bibitem[A69]{A69} W. Arveson:
{\it Subalgebras of $C^*$-algebras}, Acta Math. {\bf 123} (1969) 141-224.

\bibitem[A72]{A72} W. Arveson:
{\it Subalgebras of $C^*$-algebras, II}, Acta Math. {\bf 128}  (1972) 271-308.

\bibitem[A08]{A08} W. Arveson: {\it The noncommutative Choquet boundary}, J. Amer. Math. Soc. {\bf 21} (2008) 1065-1084.



\bibitem[BB07]{BB07}
J.A. Ball, V. Bolotnikov:
\textit{Interpolation in the noncommutative Schur-Agler class},
{J. Operator Theory} {58} (2007) 83--126.

\bibitem[DK15]{DK15}
K.R. Davidson, M. Kennedy:
\textit{The Choquet boundary of an operator system},
    Duke Math. J. 164  (2015) 2989--3004.
    
\bibitem[DK+]{DK+}
K.R. Davidson, M. Kennedy:
\textit{Noncommutative Choquet Theory},
    preprint \url{https://arxiv.org/pdf/1905.08436.pdf}.

\bibitem[DDSS17]{DDSS17} K.R. Davidson, A. Dor-On, O. Shalit, B. Solel: \textit{Dilations, inclusions of matrix convex sets, and
completely positive maps},
Internat. Math. Res. Notices  {13} (2017) 4069--4130.

\bibitem[DM05]{DM05}
M.A. Dritschel, S.A. McCullough:
\textit{Boundary representations for families of representations of operator algebras and spaces},
J. Operator Theory 53  (2005) 159--168.


\bibitem[EW97]{EW97}
E.G. Effros, S. Winkler:
\textit{Matrix convexity: operator analogues of the bipolar and Hahn-Banach theorems},
{J. Funct. Anal.} {144}  (1997)  117--152.

\bibitem[E18]{E18} E. Evert: {\it Matrix convex sets without absolute extreme points}, Linear Algebra Appl. {\bf 537} (2018) 287-301.

\bibitem[EH19]{EH19}
E. Evert, J.W. Helton:
\textit{Arveson extreme points span free spectrahedra}, 
Math. Ann. {\bf 375}, 629–653 (2019). \url{https://doi.org/10.1007/s00208-019-01858-9}.

\bibitem[EHKM18]{EHKM18}
E. Evert, J.W. Helton, I. Klep, S. McCullough:
{\it Extreme points of matrix convex sets, free spectrahedra and dilation theory},
J. of Geom. Anal. {\bf 28} (2018) 1373-1498.

\bibitem[F00]{F00}
 D.R. Farenick:
 {\it Extremal matrix states on operator systems,}
  J. London Math. Soc. {61} (2000) 885--892.

\bibitem[F04]{F04}
 D.R. Farenick:
 {\it Pure matrix states on operator systems},
  Linear Algebra Appl. {393} (2004) 149--173.

\bibitem[FNT17]{FNT17} T. Fritz, T. Netzer, A. Thom: {\it Spectrahedral Containment and Operator Systems with Finite-dimensional Realization}, SIAM J. Appl. Algebra Geom. {1} (2017) 556--574.

\bibitem[FHL18]{FHL18}
A.H. Fuller, M. Hartz, M. Lupini:
{\it Boundary representations of operator spaces, and compact rectangular matrix convex sets}, J. Operator Theory {79} (2018) 139--172.


\bibitem[Ham79]{Ham79}
M. Hamana: \textit{Injective envelopes of operator systems}, Publ. Res. Inst. Math. Sci. 15 (1979) 773--785.


\bibitem[HKM13]{HKM13} J.W. Helton, I. Klep, S. McCullough:
\textit{The matricial relaxation of a linear matrix inequality},
 {Math. Program.} {138} (2013) 401--445.

\bibitem[HKM17]{HKMjems}
  J.W. Helton, I. Klep, S. McCullough:
{\it The tracial Hahn-Banach theorem, polar duals, matrix convex sets, and projections of free spectrahedra}, J. Eur. Math. Soc. {6} (2017) 1845--1897.
 
 \bibitem[HKV+]{HKV+}
 J.W. Helton, I. Klep, J. Vol\v{c}i\v{c}:
 \textit{Factorization of noncommutative polynomials and nullstellens\"{a}tze for the free algebra}, preprint \url{https://arxiv.org/pdf/1907.04328.pdf}.
 

\bibitem[HM12]{HM12}
 J.W. Helton, S. McCullough:
\textit{Every free basic convex semi-algebraic set has an LMI representation},
{Ann. of Math. (2)} {176} (2012) 979--1013.

\bibitem[JKMMP20]{JKMMP20}
M. Jury, I. Klep, M. Mancuso, S. McCullough, J.E. Pascoe
\textit{Noncommutative Partial Convexity Via $\Gamma$-Convexity}, J. of Geom. Anal. (2020)  \url{https://doi.org/10.1007/s12220-020-00387-1}. 


\bibitem[KVV14]{KVV14}
D. Kalyuzhnyi-Verbovetski\u\i{}, V. Vinnikov:
{\it
Foundations of free noncommutative function theory},
Amer. Math. Soc., 2014.


\bibitem[Kls14]{Kls14}
C. Kleski:
\textit{Boundary representations and pure completely positive maps},
 J. Operator Theory 71 (2014)  45--62.
 
  \bibitem[K+]{K+}
T.-L. Kriel:
 \textit{An introduction to matrix convex sets and free spectrahedra}, preprint \url{https://arxiv.org/abs/1611.03103}.

\bibitem[MS11]{MS11}
P.S. Muhly, B. Solel:
\textit{Progress in noncommutative function theory},
{Sci. China Ser. A} {54} (2011) 2275--2294.

\bibitem[MS98]{MS98}
P.S. Muhly, B. Solel:
 {\it An algebraic characterization of boundary representations,} Nonselfadjoint operator algebras, operator theory, and related topics, 189--196, Oper. Theory Adv. Appl., 104, Birkhäuser, Basel, 1998.



\bibitem[PP+]{PP+}
B. Passer, V. Paulsen: 
{\it Matrix range characterizations of operator system properties}, preprint \url{https://arxiv.org/pdf/1912.06279.pdf}.

\bibitem[PSS18]{PSS18}
B. Passer, O. Shalit, B. Solel:
{\it Minimal and maximal matrix convex sets}, J. Funct. Anal. {\bf 274} (2018) 3197-3253.



\bibitem[Pop10]{Pop10}
G. Popescu:
\textit{Free holomorphic automorphisms of the unit ball of $B(H)^n$},
    {J. reine angew. Math.} {638} (2010) 119--168.



\bibitem[SSS18]{SSS18}
G. Salomon, O.M. Shalit, E. Shamovich: 
\textit{Algebras of bounded noncommutative analytic functions on subvarieties of the noncommutative unit ball}
, Trans. Amer. Math. Soc. 370 (2018) 8639–8690.




\bibitem[Voi10]{Voi10}
D.-V. Voiculescu:
\textit{Free analysis questions II: The Grassmannian completion and the series
expansions at the origin}, {J. reine angew. Math.} {645} (2010) 155--236.




\bibitem[WW99]{WW99}
C. Webster, S. Winkler:
 \textit{The Krein-Milman theorem in operator convexity},
 {Trans. Amer. Math. Soc.} {351} (1999)  307--322.


\bibitem[Zal17]{Zal17}
 A. Zalar:
 \textit{Operator Positivstellens\"atze for noncommutative polynomials positive on matrix convex sets}, J. Math. Anal. Appl.
 445 (2017) 32--80.
 


\end{thebibliography}
\end{document}